\documentclass[oneside]{amsart}
\pdfoutput=1
\usepackage[a4paper]{geometry}
\usepackage[T1]{fontenc}
\usepackage[utf8]{inputenc}
\usepackage{lmodern}
\usepackage{pdflscape}
\usepackage{diagbox}

\usepackage{xcolor} %

\definecolor{col1}{RGB}{254,97,0}
\definecolor{col2}{RGB}{100,143,255}
\definecolor{col3}{RGB}{120, 94, 240}
\definecolor{col4}{RGB}{220, 38, 127}
\definecolor{col5}{RGB}{255, 176, 0}

\usepackage{tikz}
\tikzset{vertex/.style = {shape=circle,draw,fill,minimum size=3pt, inner sep=0pt}}

\usepackage{amsmath,amssymb}
\usepackage{amsthm}
\usepackage{subcaption}
\usepackage{hyperref}

\newtheorem{theorem}{Theorem}[section]
\newtheorem{lemma}[theorem]{Lemma}
\newtheorem{definition}[theorem]{Definition}
\newtheorem{corollary}[theorem]{Corollary}
\newtheorem{proposition}[theorem]{Proposition}
\newtheorem{conjecture}[theorem]{Conjecture}

\theoremstyle{remark}
\newtheorem{remark}[theorem]{Remark}

\newcommand{\bb}[1]{{\overline{#1}}}

\newcommand{\Q}{\mathbb{Q}}

\newcommand{\Z}{\mathbb{Z}}
\newcommand{\tQ}{\widetilde{\Q}}

\DeclareMathOperator{\Tr}{Tr}

\newcommand{\A}{{\operatorname{U\Gamma}}}
\newcommand{\PA}{{\operatorname{\Gamma}}}
\DeclareMathOperator{\SG}{{\mathbb S}}
\newcommand{\GL}{{\operatorname{GL}}}
\newcommand{\Aut}{{\operatorname{Aut}}}
\newcommand{\UAut}{{\operatorname{UAut}}}
\newcommand{\Out}{{\operatorname{Out}}}
\newcommand{\sign}{{\operatorname{sign}}}

\newcommand{\Ind}{\operatorname{Ind}}
\newcommand{\ch}{\operatorname{ch}}

\newcommand{\parts}{\vdash}

\DeclareMathOperator{\im}{im}
\DeclareMathOperator{\id}{id}

\DeclareMathOperator{\ee}{\widehat{e}}

\DeclareMathOperator{\rank}{\mathrm{rk}}

\newcommand{\MG}{\mathcal{MG}}
\newcommand{\F}{\mathcal{F}}

\hypersetup{
  pdfauthor={Michael Borinsky, Jos Vermaseren},
  pdftitle={The Sn-equivariant Euler characteristic of the moduli space of graphs},
  pdfsubject={},
  pdfkeywords={},
  linkcolor  = col3,
  citecolor  = col4,
  urlcolor   = col5,
  colorlinks = true,
}

\title[The ${\mathbb S}_n$-equivariant Euler characteristic of the moduli space of graphs]{The ${\mathbb S}_n$-equivariant Euler characteristic of\\ the moduli space of graphs}

\author{Michael Borinsky \and Jos Vermaseren}

\address{
Michael Borinsky\\
Institute for Theoretical Studies\\
ETH Z\"urich\\
8092 Zürich, Switzerland
}
\address{
Jos Vermaseren\\
Theory Department\\
Nikhef\\
1098 XG Amsterdam, Netherlands
}

\begin{document}

\begin{abstract}
We prove a formula for the ${\mathbb S}_n$-equivariant Euler characteristic of the moduli space of graphs $\mathcal{MG}_{g,n}$. Moreover, we prove that the rational ${\mathbb S}_n$-invariant cohomology of $\mathcal{MG}_{g,n}$ stabilizes for large $n$. That means, if $n \geq g \geq 2$, then there are isomorphisms $H^k(\mathcal{MG}_{g,n};\mathbb{Q})^{{\mathbb S}_n} \rightarrow H^k(\mathcal{MG}_{g,n+1};\mathbb{Q})^{{\mathbb S}_{n+1}}$ for all $k$.
\end{abstract}
\maketitle

\section{Introduction}
A graph $G$ is a finite, at most one-dimensional CW complex. It has \emph{rank} $g$ if its fundamental group is a free group of rank $g$: $\pi_1(G) \cong F_g$. %
Here, graphs shall be \emph{admissible}, that means they do not have vertices of degree $0$ or $2$.
Univalent vertices have a special role and are called \emph{legs} (often also \emph{hairs}, \emph{leaves} or \emph{marked points}).
Similarly, we will reserve the name \emph{edge} to 1-cells that are only incident to non-leg vertices.
Legs of graphs are uniquely labeled by integers $\{1,\ldots,n\}$.

A \emph{metric graph} is additionally equipped with a length $\ell(e) \geq 0$ for each edge $e$ such that all edge lengths sum to one, $\sum_{e} \ell(e)=1$.
Fix $g,n$ such that $g >0$ and $2g-2+n > 0$. 
The moduli space of graphs, $\MG_{g,n}$,
is the space of isometry classes of metric graphs of rank $g$ with $n$ legs, except for graphs  that have a cycle in which all edges have length $0$.
It inherits the topology from the metric and by identifying each graph that has an edge $e$ of length $0$ with the respective graph where $e$ is collapsed.

The moduli space of graphs was introduced in \cite{CV},
where it was shown that $\MG_{g,0}$ serves as a rational classifying space for
$\Out(F_g)$, the outer automorphism group of the free group of rank $g$.
Further,
$\MG_{g,n}$ can be seen as the moduli space of \emph{pure tropical curves}~\cite{Abramovich} and it is relevant for  \emph{Feynman amplitudes}, central objects in quantum field theory~\cite{Berghoff:2017dyq}. 

A partition $\lambda \parts n$ gives rise to both an irreducible representation $V_\lambda$ of the symmetric group $\SG_n$ and a Schur polynomial $s_\lambda$, a symmetric polynomial in $\Lambda_n = \Q[x_1,\ldots,x_n]^{\SG_n}$. 
The symmetric group acts on the cohomology of $\MG_{g,n}$ by permuting the leg-labels. 
So, $H^k(\MG_{g,n};\Q)$ is an $\SG_n$-representation that we can decompose into irreducibles, i.e.~there are integers $c_{g,\lambda}^k$ such that
\begin{align*} H^k(\MG_{g,n};\Q)\, \cong\, \bigoplus_{\lambda \parts n} c_{g,\lambda}^k V_\lambda. \end{align*}
The multiplicities $c_{g,\lambda}^k$ are known explicitly if $g\leq 2$ \cite{conant2016assembling}.
The $\SG_n$-equivariant Euler characteristic of $\MG_{g,n}$ is the following symmetric polynomial,
\begin{align} \label{eq:Sndef} e_{\SG_n}(\MG_{g,n}) \,=\,\sum_{\lambda \parts n} s_\lambda \sum_{k} (-1)^k c_{g,\lambda}^k. \end{align}

Our first main result is an effective formula for $e_{\SG_n}(\MG_{g,n})$. It is stated as Theorem~\ref{thm:comp}. 
Its proof in Section~\ref{sec:SGequivariant} is based on prior work by Vogtmann and the first author~\cite{BV2}.

Analogous formulas exist, e.g., for the $\SG_n$-equivariant Euler characteristic of $\mathcal M_{g,n}$~\cite{gorsky2014equivariant} and for the $\SG_n$-equivariant Euler characteristic of  the moduli space of stable tropical curves \cite{chan2019s_n}. 
The latter moduli space is a compactification of $\mathcal {MG}_{g,n}$ and its cohomology injects into the cohomology of $\mathcal M_{g,n}$ \cite{CGP}. 

There is another important invariant of moduli spaces such as $\MG_{g,n}$, which is in general a rational number: the \emph{virtual} Euler characteristic. It has more convenient properties while studying maps between moduli spaces than the \emph{classical} Euler characteristic, the alternating sum of the Betti numbers.  The virtual Euler characteristic of $\MG_{g,n}$ was studied, for instance, in \cite{BV}. The quantitative relation between the virtual and classical Euler characteristic of $\MG_{g,n}$ is discussed in \S \ref{sec:asy}.

To obtain the classical Euler characteristic, $e(\MG_{g,n})=\sum_{k} (-1)^k \dim H^k(\MG_{g,n};\Q)$, we replace $s_\lambda$ with $\dim V_\lambda$ in eq.~\eqref{eq:Sndef}. Another specialization of $e_{\SG_n}(\MG_{g,n})$ is the Euler characteristic for the \emph{$\SG_n$-invariant} part of $\MG_{g,n}$'s cohomology, denoted as $e(\MG_{g,n}^{\SG_n})$. We obtain it from $e_{\SG_n}(\MG_{g,n})$ by setting $s_{(n)} =1$, (which corresponds to the trivial representation), and $s_{\lambda} = 0$ for all $\lambda \neq (n)$:
\begin{align*} e(\MG_{g,n}^{\SG_n})\,=\,\sum_{k} (-1)^k \dim H^k(\MG_{g,n};\Q)^{\SG_n}\,=\, \sum_k (-1)^k c_{g,(n)}^k. \end{align*}

We denote $\PA_{g,n}$ as the group of homotopy classes of self-homotopy equivalences of a connected graph $G$ of rank $g$ with $n$ legs that fix the legs point-wise. I.e.~$\PA_{g,n}=\pi_0(\mathrm{HE}(G,\mathrm{fix}~\partial G))$. We have $\PA_{g,0}\cong \Out(F_g)$ and $\PA_{g,1}\cong \Aut(F_g)$. The moduli space $\MG_{g,n}$ is a rational classifying space of $\PA_{g,n}$, implying that the rational cohomology of $\MG_{g,n}$ is the same as that of $\PA_{g,n}$  (see,~e.g.,~\cite{conant2016assembling}).

Kontsevich made the observation that the homology of $\MG_{g,n}$ is computed by what he termed the \emph{Lie graph complex}~\cite{Ko1}. In Section~\ref{sec:forest}, we utilize a related concept known as the \emph{forested graph complex}, which was introduced in \cite{CoVo}, for the purpose of computing this homology.

In his work, Kontsevich also proposed an \emph{odd} version of the Lie graph complex. This differs from the original graph complex by the notion of the graphs' orientation. Detailed discussions regarding this are given in Sections~\ref{sec:forest} and \ref{sec:orient}. He also indicated that, in the absence of legs, this odd graph complex computes the cohomology of $\Out(F_g)$ with coefficients in $\widetilde{\mathbb Q}$, the representation of $\Out(F_g)$ that results from the composition of the natural map $\Out(F_g) \rightarrow \GL_g(\Z)$ with the determinant.
A formula for this `odd' Euler characteristic, denoted as $e^\mathrm{odd}(\Out(F_g)) = \sum_{k} (-1)^k \dim H^k(\Out(F_g),\widetilde \Q)$, was previously presented in~\cite{BV2}. Similar formulas can be obtained for all groups $\PA_{g,n}$ due to the existence of the natural surjection $\PA_{g,n}\rightarrow \Out(F_g)$ that comes from forgetting the legs.
The odd Euler characteristic is equivalent to the Euler characteristic of the moduli space of graphs $\MG_{g,n}$ that is fibered by the local system resulting from the $\det$-representation of $\Out(F_n)$ as pointed out in \cite{Ko1}.
Here, we will prove a formula for the $\SG_n$-equivariant odd Euler characteristic, denoted as $e^\mathrm{odd}_{\SG_n}(\MG_{g,n})$. It is defined analogously to $e_{\SG_n}(\MG_{g,n})$ in eq.~\eqref{eq:Sndef} with the coefficients in $\tQ$ instead of $\Q$.

Morita, Sakasai, and Suzuki previously calculated the Euler characteristic, $e(\MG_{g,0}) = e(\Out(F_g))$ for $g\leq 11$~\cite{morita2015integral}. 
The outcome of recent work of Vogtmann and the first author was the asymptotic formula, $e(\Out(F_g)) \sim - e^{-\frac14} {\left(g/e\right)^g/(g \log g)^2}$ for large $g$, implying that the total dimension of the cohomology of $\Out(F_g)$ grows super-exponentially with $g$~\cite{BV2}. In this work, we enhance this result with explicit computations.

We leverage our effective formulas with the \texttt{FORM} programming language by the second author~\cite{Vermaseren:2000nd}, for instance, to compute $e(\Out(F_g))$ for all $g \leq 100$. The resulting values of $e(\Out(F_g))$ and $e^\mathrm{odd}(\Out(F_g))$ for $g\leq 15$ were published with \cite{BV2}.
The ancillary files to the arXiv version of the present article include these numbers for all $g\leq 100$, together with tables of $e(\MG_{g,n})$, $e^\mathrm{odd}(\MG_{g,n})$, $e(\MG_{g,n}^{\SG_n})$, $e^\mathrm{odd}(\MG_{g,n}^{\SG_n})$ for all $g+n \leq 60$, and the polynomials $e_{\SG_n}(\MG_{g,n})$, $e_{\SG_n}^\mathrm{odd}(\MG_{g,n})$ for all $g+n \leq 30$.
Subsets of these data are given in Tables~\ref{tab:eeven}--\ref{tab:eEodd}.
In Section~\ref{sec:form}, we comment on our implementation of Theorem \ref{thm:comp} in \texttt{FORM}, which is included in the ancillary files.
Based on our computed data, we give some empirical observations on the asymptotic growth rate of the Euler characteristic $e(\MG_{g,n})$ and its variations for large $g$ in Section~\ref{sec:asy}.

Our data in Tables~\ref{tab:eEeven}--\ref{tab:eEodd} exhibit an interesting pattern. The $\SG_n$-invariant Euler characteristics of $\MG_{g,n}$ \emph{stabilize} for large $n$.
The search for an explanation of this pattern led to our second main finding: The $\SG_n$-invariant cohomologies $H^\bullet(\MG_{g,n};\Q)^{\SG_n}$ and  $H^\bullet(\MG_{g,n};\tQ)^{\SG_n}$ \emph{stabilize} for large $n$. That means if  $n \geq g \geq 2$ and $\Q_\rho \in \{\Q,\tQ\}$, then there are  isomorphisms $H^k(\MG_{g,n};\Q_\rho)^{\SG_n}\rightarrow H^k(\MG_{g,n+1};\Q_\rho)^{\SG_n}$ for all $k$.
 This statement is proven as Theorem~\ref{thm:stab} in Section~\ref{sec:stable}, where we also comment on an analogy to the cohomology of the braid group. 
Theorem~\ref{thm:stab} complements previous stability results for $\MG_{g,n}$ (see \cite{conant2016assembling} and the references therein). 
Here, in contrast to these previous results, the range of stabilization does not depend on the cohomological degree. 

\section*{Acknowledgements}
MB owes many thanks to Karen Vogtmann for many discussions and related joint work.  He is indebted to S{\o}ren Galatius and Thomas Willwacher
for valuable insights that led to the proof of Theorem~\ref{thm:stab} and thanks Benjamin Br\"uck and Nathalie Wahl for helpful discussions.
We thank the anonymous referee for valuable comments and suggestions, in particular for spotting a gap in an earlier argument for Proposition~\ref{prop:oddeven}.
MB also thanks the Institute of Advanced Studies, Princeton US, where parts of this work were completed, for hospitality.
  MB was supported by Dr.\ Max Rössler, the Walter Haefner Foundation and the ETH Zürich Foundation.  %

\section{The \texorpdfstring{$\SG_n$}{Sn}-equivariant Euler characteristic of \texorpdfstring{$\MG_{g,n}$}{MGgn}}
\label{sec:SGequivariant}
\subsection{Forested graph complexes}
\label{sec:forest}
In this section we will describe chain complexes that compute the homology of $\MG_{g,n}$. These chain complexes are generated by certain graphs.

A subgraph is a subcomplex of a graph that consists of all its vertices, its non-edge 1-cells and a subset of its edges.
A subforest is an acyclic subgraph. A pair $(G,\Phi)$ of a graph $G$ and a subforest $\Phi \subset G$ is 
a \emph{forested graph}.
We write $|\Phi|$ for the number of edges in the forest $\Phi$.
Figure~\ref{fig:forestedgraphs} shows some examples of forested graphs with different rank, forest edge and leg numbers.
The forest edges are drawn thicker and in blue. Legs are drawn as labeled half-edges.
A $(+)$-marking $\sigma_\Phi$ of $(G,\Phi)$ is an ordering of the forest edges, i.e.~a bijection $\sigma_\Phi: E_\Phi \rightarrow \{1,\ldots,|\Phi|\}$.
A $(-)$-marking $(\sigma_\Phi,\sigma_{H_1(G,\Z)})$ of $(G,\Phi)$ is such an ordering of the forest edges $\sigma_\Phi$
together with a basis for the first homology of the graph, i.e.~a bijection
$\sigma_{H_1}:H_1(G,\Z) \rightarrow \Z^{h_1(G)}$, where $h_1(G) = \rank H_1(G,\Z)$ is the rank of $H_1(G,\Z)$ or equivalently the rank of $G$.
\begin{definition}
\label{def:fgc}
For given $g,n \geq 0$ with $2g-2+n \geq 0$, we define two $\Q$-vector spaces:
\begin{itemize}
\item
$\F_{g,n}^{+}$ is generated by tuples $(G,\Phi,\sigma_\Phi)$ of a connected admissible forested graph $(G,\Phi)$ of rank $g$ with $n$ legs,
which is 
$(+)$-marked with $\sigma_\Phi$, modulo the relation
\begin{align*} (G,\Phi,\pi \circ \sigma_\Phi) \sim \sign(\pi) \cdot (G,\Phi,\sigma_\Phi) & \text{ for all } \pi \in \SG_{|\Phi|}, \end{align*}
and modulo isomorphisms of $(+)$-marked forested graphs.
\item
$\F_{g,n}^{-}$ is generated by tuples $(G,\Phi,\sigma_\Phi,\sigma_{H_1})$ of a connected admissible forested graph $(G,\Phi)$ of rank $g$ with $n$ legs,
which is
$(-)$-marked with $\sigma_\Phi,\sigma_{H_1}$ modulo the relation
\begin{gather*} (G,\Phi,\pi \circ \sigma_\Phi, \rho_{H_1} \circ \sigma_{H_1}) \sim \sign(\pi) \cdot \det \rho_{H_1}\cdot (G,\Phi,\sigma_\Phi) \\
 \text{ for all } \pi \in \SG_{|\Phi|} \text{ and } \rho_{H_1} \in \GL_{h_1(G)}(\Z) \end{gather*}
and modulo isomorphisms of $(-)$-marked forested graphs.
\end{itemize}
\end{definition}

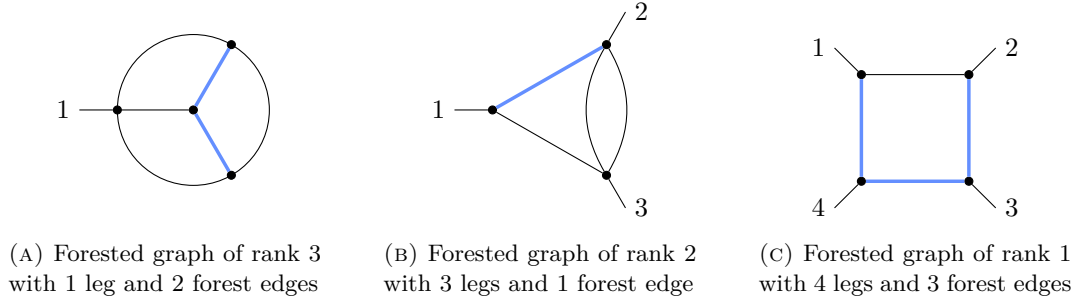
\begin{figure}
    \begin{subfigure}[b]{0.28\textwidth}
        \centering
        \begin{tikzpicture} \node[vertex] (v) at (0:0) {}; \draw (v) circle (1); \foreach \angle in {60,180,300} { \draw (\angle:1) node[vertex] (v\angle) {}; } \draw[col2, line width=1.3pt] (v) -- (v60); \draw[col2, line width=1.3pt] (v) -- (v300); \draw (v) -- (v180); \draw (v180) -- ++(-.5,0) node[left] {$1$}; \draw[white] (v300) -- ++(300:.5) node[right] {$3$}; \end{tikzpicture}
        \caption{Forested graph of rank $3$ with $1$ leg and $2$ forest edges}
        \label{fig:fg312}
    \end{subfigure}
    \qquad
    \begin{subfigure}[b]{0.28\textwidth}
        \centering
        \begin{tikzpicture} \foreach \angle in {60,180,300} { \draw (\angle:1) node[vertex] (v\angle) {}; } \draw (v180) -- (v300); \draw[col2, line width=1.3pt] (v180) -- (v60); \draw (v60) to[bend right=30] (v300); \draw (v60) to[bend left=30] (v300); \draw (v180) -- ++(180:.5) node[left] {$1$}; \draw (v60) -- ++(60:.5) node[right] {$2$}; \draw (v300) -- ++(300:.5) node[right] {$3$}; \end{tikzpicture}
        \caption{Forested graph of rank $2$ with $3$ legs and $1$ forest edge}
        \label{fig:fg231}
    \end{subfigure}
    \qquad
    \begin{subfigure}[b]{0.28\textwidth}
        \centering
        \begin{tikzpicture} \foreach \angle in {45,135,225,315} { \draw (\angle:1) node[vertex] (v\angle) {}; } \draw (v45) -- ++(45:.5) node[right] {$2$}; \draw (v135) -- ++(135:.5) node[left] {$1$}; \draw (v225) -- ++(225:.5) node[left] {$4$}; \draw (v315) -- ++(315:.5) node[right] {$3$}; \draw[col2, line width=1.3pt] (v135) -- (v225) -- (v315) -- (v45); \draw (v45) -- (v135); \end{tikzpicture}
        \caption{Forested graph of rank $1$ with $4$ legs and $3$ forest edges}
        \label{fig:fg143}
    \end{subfigure}
    \caption{Examples of forested graphs}
    \label{fig:forestedgraphs}
\end{figure}

We will discuss the relations imposed in this definition and how they give rise to two different ways to impose \emph{orientations} on graphs in more detail in Section~\ref{sec:orient}.

In what follows, we will often cover both the $(+)$ and $(-)$-cases analogously in the same sentence using the $\pm$-notation.
Notice that the vector spaces $\F_{g,n}^\pm$ are defined for a slightly larger range of pairs $(g,n)$ then the moduli spaces $\MG_{g,n}$. This generalization makes the generating functions for the Euler characteristics in Section~\ref{sec:genfun} easier to handle.
There are no (admissible) graphs of rank one without legs, so 
$\F_{1,0}^{+} = \F_{1,0}^{-} = \emptyset$, 
but there is one admissible graph of rank $0$ with two legs and no vertices or edges. 
It just consists of a 1-cell that connects the legs.

\begin{lemma}
\label{lmm:finite}
For all $g,n \geq 0$ with $2g-2+n > 0$,
the vector spaces $\F_{g,n}^{\pm}$ 
are finite dimensional.
\end{lemma}
\begin{proof}
Let $G$ be a connected admissible graph of rank $g$ with $n$ legs, $|E_G|$ edges, and $|V_G|$ vertices. The Euler characteristic of such a graph equals  $|V_G| - |E_G| = 1-g$. 
If we exclude the aforementioned graph with two legs and no vertices or edges, then each leg of $G$ is incident to a vertex.
Each vertex has at most 3 incident edges or legs, so $3|V_G|\leq n+2|E_G|$.
It follows that $|E_G| \leq 3g-3+n$. There are only finitely many isomorphism classes of graphs of bounded edge number.
Hence, there are also only finitely many isomorphism classes of forested graphs of rank $g$ with $n$ legs. The relations in Definition~\ref{def:fgc} guarantee that each such isomorphism class contributes at most one generator to $\F_{g,n}^{\pm}$.
\end{proof}

The vector spaces $\F_{g,n}^{\pm}$ are graded by the number of edges in the forest. Let 
$C_k(\F_{g,n}^{\pm})$ be the respective subspace restricted to generators with $k$ forest edges.
These spaces form a chain complex, so we will refer to $\F_{g,n}^\pm$ as the \emph{forested graph complex} with $(\pm)$ orientation.
We will not explicitly state the boundary maps, $\partial_k: C_k(\F_{g,n}^{\pm})\rightarrow C_{k-1} (\F_{g,n}^{\pm})$ here (see,~\cite{CoVo}), as knowledge of the dimensions of these chain groups suffices for our Euler characteristic considerations.

The following theorem is a consequence of the works of Culler--Vogtmann~\cite{CV}, Kontsevich~\cite{Ko1,Ko2} and Conant--Vogtmann \cite{CoVo} (see in particular \cite[Sec.~3.1--3.2]{CoVo}):

\begin{theorem}
\label{thm:forested_graph_complex}
For $g > 0$,  $n\geq 0$ and $2g-2+n > 0$,
the chain complexes $\F_{g,n}^{+}$ ($\F_{g,n}^{-}$)
compute the homology of $\MG_{g,n}$ with trivial coefficients $\Q$ (with twisted coefficients $\tQ$). 
The $\SG_n$-action on $H^\bullet(\MG_{g,n},\Q)$ ($H^\bullet(\MG_{g,n},\tQ)$) descents obviously to $\F_{g,n}^+$ ($\F_{g,n}^-$) by permuting  leg-labels.
\end{theorem}

In what follows, we will use this theorem to prove a formula for 
$e_{\SG_n}(\MG_{g,n})$ and $e^\mathrm{odd}_{\SG_n}(\MG_{g,n})$.

\subsection{Equivariant Euler characteristics}

Recall that a permutation $\pi \in \SG_n$ factors uniquely as a product of disjoint cycles. 
If the orders of these cycles are $\lambda_1, \ldots, \lambda_\ell$, 
then %
$(\lambda_1,\ldots,\lambda_\ell)$ is a partition of $n$ called the \emph{cycle type} of $\pi$. %
For such a permutation $\pi \in \SG_n$, we define the \emph{power sum symmetric polynomial},
$p^\pi = p_{\lambda_1} \cdots p_{\lambda_\ell} \in \Lambda_n$
with $p_k=\sum_{i=1}^n x_i^k$. %
The  \emph{Fr\"obenius characteristic} is a symmetric polynomial associated uniquely to a $\SG_n$-representation $V$. It is defined by 
\begin{align*}  \ch(\chi_{V}) = \frac{1}{n!} \sum_{\pi \in \SG_n} \chi_{V}(\pi) p^\pi, \end{align*}
where $\chi_V$ is the \emph{character} associated to $V$ (see,~e.g.,~\cite[\S7.18]{stanley1997enumerative2}). 
The
\emph{$\SG_n$-equivariant Euler characteristic} of $\F_{g,n}^\pm$ is the alternating sum over the Fr\"obenius characteristics of $H_k(\F_{g,n}^\pm;\Q)$. Note that this is consistent with the definition of $e_{\SG_n}(\MG_{g,n})$ in eq.~\eqref{eq:Sndef} since $\ch(\chi_{V_\lambda}) = s_\lambda$.

As the spaces $\F_{g,n}^\pm$ are finite by  Lemma~\ref{lmm:finite},
we can compute the equivariant Euler characteristic 
on the chain level:
\begin{proposition}
\label{prop:equivC}
The $\SG_n$-equivariant Euler characteristic of $\F_{g,n}^\pm$ is given by 
\begin{align*} e_{\SG_n}(\F_{g,n}^\pm) &= \sum_k (-1)^k \ch(\chi_{C_k(\F^\pm_{g,n})}), \end{align*}
where $\ch(\chi_{C_k(\F^\pm_{g,n})})$
is the Fr\"obenius characteristic of $C_k(\F^\pm_{g,n})$ as an $\SG_n$-representation,
\begin{align*} \ch(\chi_{C_k(\F^\pm_{g,n})}) = \frac{1}{n!} \sum_{\pi \in \SG_n} \chi_{C_k(\F^\pm_{g,n})}(\pi) p^\pi, \end{align*}
and $\chi_{C_k(\F^\pm_{g,n})}$ 
is the \emph{character} associated to $C_k(\F^\pm_{g,n})$.
\end{proposition}

\begin{corollary}
\label{cor:MGF}
For $g > 0$,  $n\geq 0$ and $2g-2+n > 0$,
\begin{align*} e_{\SG_n}(\MG_{g,n}) &= e_{\SG_n}(\F_{g,n}^+) &\text{ and }&& e^\mathrm{odd}_{\SG_n}(\MG_{g,n}) &= e_{\SG_n}(\F_{g,n}^-) . \end{align*}
\end{corollary}
\begin{proof}
Follows directly from Theorem~\ref{thm:forested_graph_complex}.
\end{proof}

The other discussed Euler characteristics of $\MG_{g,n}$ can 
be obtained by evaluating the polynomials $e_{\SG_n}(\F_{g,n}^\pm)$ 
for certain values of $\bb p=p_1,p_2,\ldots$.
We write $\bb p = \bb 1$ for the specification $p_1=1$, $p_2=p_3=\ldots=0$
and $\bb p = \boldsymbol 1$ for the specification $p_1=p_2=\ldots=1$.
\begin{proposition}
\label{prop:MGF-special}
For $g > 0$,  $n\geq 0$ and $2g-2+n > 0$,
\begin{align*} e(\MG_{g,n}) &= n! \cdot e_{\SG_n}(\F_{g,n}^+) |_{\bb p = \bb 1}, & e^\mathrm{odd}(\MG_{g,n}) &= n! \cdot e_{\SG_n}(\F_{g,n}^-) |_{\bb p = \bb 1},\\
e(\MG_{g,n}^{\SG_n}) &= e_{\SG_n}(\F_{g,n}^+) |_{\bb p = \boldsymbol 1}, & e^\mathrm{odd}(\MG_{g,n}^{\SG_n}) &= e_{\SG_n}(\F_{g,n}^-) |_{\bb p = \boldsymbol 1}. \end{align*}
\end{proposition}
\begin{proof}
Recall the relationship between the character of the 
$\SG_n$-representation $C_k(\F_{g,n}^\pm)$ and the symmetric 
polynomials $e_{\SG_n}(\F_{g,n}^\pm)$ from 
Proposition~\ref{prop:equivC}.
Observe that
$\chi_{V}(\id) = \dim V$, where $\id$ is the trivial permutation.
To verify the first two equations, observe that substituting $p_1=1$ and $p_2=p_3=\ldots=0$,
in the formula for 
$ \ch(\chi_{C_k(\F^\pm_{g,n})}) $
of Proposition~\ref{prop:equivC}
amounts to restricting the sum over $\pi$ to $\pi = \id$,
as only the trivial permutation has cycle-type $(1,1,\ldots,1)$.
For the second line recall that 
$\frac{1}{n!} \sum_{\pi \in \SG_n} \chi_{V}(\pi) = \dim V^{\SG_n},$
where $V^{\SG_n}$ is the $\SG_n$-invariant subspace of $V$.
Additionally, use Corollary~\ref{cor:MGF}.
\end{proof}

To explicitly compute $e_{\SG_n}(\F^\pm_{g,n})$,
we hence need to compute the Fr\"obenius characteristic of the chain groups and to compute those we need to know the character $\chi_{C_k(\F^\pm_{g,n})}$.
To compute these characters, we need a more detailed description of the chain groups $C_k(\F^\pm_{g,n})$ that are generated only by 
\emph{orientable forested graphs}.

\subsection{Orientable forested graphs}
\label{sec:orient}

Roughly, the relations in Definition~\ref{def:fgc} have the effect that a chosen ordering of the forest edges or a chosen homology basis, \emph{only matters up to its sign or its orientation}. We can therefore think of the generators as \emph{oriented} forested graphs.
Due to the relations, every forested graphs gives rise to at most one generator of $\F_{g,n}^{+}$ or $\F_{g,n}^{-}$.
However, not every forested graph is \emph{orientable} in this way, i.e.~gives rise to a generator. 
Even though Definition~\ref{def:fgc} only involves connected forested graphs, we will define the notion of orientability here also for disconnected graphs. We will need it later.

Let $\Aut(G,\Phi)$ be the group of automorphisms of the forested graph $(G,\Phi)$.
Each automorphism is required to fix the leg-labels and the forest.
For instance, the graph in Figure~\ref{fig:fg312} has one automorphism of order two that mirrors the graph on the horizontal axis.
The graph in Figure~\ref{fig:fg231} has one automorphism that flips the two doubled edges 
and the graph in Figure~\ref{fig:fg143} has no (leg-label-preserving) automorphisms.

Each $\alpha \in \Aut(G,\Phi)$ induces a permutation, $\alpha_\Phi:E_\Phi \rightarrow E_\Phi$ on the set of forest edges and an automorphism 
on the homology groups $\alpha_{H_0}: H_0(G,\Z)\rightarrow H_0(G,\Z)$ and $\alpha_{H_1}:H_1(G,\Z)\rightarrow H_1(G,\Z)$.
For each (possibly disconnected) forested graph $(G,\Phi)$ with an automorphism $\alpha \in \Aut(G,\Phi)$,
we define 
\begin{align*} \xi^+(G,\Phi,\alpha)= \sign(\alpha_\Phi) \text{ and } \xi^-(G,\Phi,\alpha)= \det(\alpha_{H_0}) \det(\alpha_{H_1}) \sign(\alpha_\Phi). \end{align*}
As connected graphs come with a canonical basis for $H_0(G,\Z)= \Z$ that cannot be changed by automorphisms, we always have  $\det(\alpha_{H_0}) = 1$
for them.
The other sign factors capture the signs of the relations in Definition~\ref{def:fgc} if the $(\pm)$-marking of the graph $(G,\Phi)$ is acted upon using $\alpha$ in the obvious way.
\begin{definition}
\label{def:orient}
A forested graph $(G,\Phi)$ is
\emph{$(\pm)$-orientable} if it has no
automorphism $\alpha \in \Aut(G,\Phi)$ for which $\xi^\pm(G,\Phi,\alpha)=-1$.
\end{definition}
For example, the automorphism of the forested graph in Figure~\ref{fig:fg312} 
switches two forest edges and flips the orientation of a chosen homology basis. 
Hence, the graph is $(-)$-, but not $(+)$-orientable.
The automorphism of the graph in Figure~\ref{fig:fg231} does not affect its subforest, but
flips the orientation of its homology basis. So, it is $(+)$-, but not $(-)$-orientable. 
The forested graph in Figure~\ref{fig:fg143} is both $(+)$- and $(-)$-orientable as it has no nontrivial automorphism.
There are also forested graphs that are neither $(+)$- nor $(-)$-orientable. 

\begin{proposition}
\label{prop:genorient}
The chain groups $C_k(\F^\pm_{g,n})$ are freely generated by isomorphism classes of connected $(\pm)$-orientable forested graphs of rank $g$ with $n$ legs 
and $k$ forest edges.
\end{proposition}
\begin{proof}
Definition~\ref{def:fgc} guarantees that each isomorphism class of a connected forested graph $(G,\Phi)$ contributes at most one generator, as all different $(\pm)$-markings of $(G,\Phi)$ are related via a permutation $\pi \in \SG_{|\Phi|}$ (and $\rho_{H_1} \in \GL_{h_1(G)}(\Z)$).

It remains to be proven that $(G,\Phi)$ does contribute a free generator if and only if it is $(\pm)$-orientable.  We will only prove this for the $(+)$-orientation, as the $(-)$-orientation follows analogously.
Recall that for any isomorphism class of a $(+)$-marked graph $(G,\Phi,\sigma_\Phi)$, Definition~\ref{def:fgc} imposes the relation $(G,\Phi,\pi \circ \sigma_\Phi) \sim \sign(\pi) \cdot (G,\Phi,\sigma_\Phi)$ for each $\pi \in \SG_{|\Phi|}$.
As $(G,\Phi,\sigma_\Phi)$ is given up to isomorphism, we have $(G,\Phi,\sigma_\Phi \circ \alpha_\Phi) = (G,\Phi,\sigma_\Phi)$ for
each $\alpha \in \Aut(G,\Phi)$ with $\alpha_\Phi$ the permutation that $\alpha$ induces on $\Phi$.
If there is an automorphism $\alpha \in \Aut(G,\Phi)$ such that $\alpha_\Phi$ is 
an odd permutation on $\Phi$, then the permutation $\pi:=\sigma_\Phi \circ \alpha_\Phi \circ \sigma_\Phi^{-1} \in \SG_{|\Phi|}$ is odd as well and we find 
that $ (G,\Phi,\sigma_\Phi)=(G,\Phi,\sigma_\Phi\circ \alpha_\Phi) = (G,\Phi, \pi \circ \sigma_\Phi) \sim - (G,\Phi, \sigma_\Phi)\sim 0$.
If there is no such automorphism, then for any $\alpha\in \Aut(G,\Phi)$,
$\sign (\sigma_\Phi \circ \alpha_\Phi \circ \sigma_\Phi^{-1})=1$. Hence, the sign of a $(+)$-orientable and $(+)$-marked forested graph is fixed by any automorphism which therefore gives a generator.
\end{proof}

\subsection{The action of \texorpdfstring{$\SG_n$}{Sn} on orientable graphs}
The next step for the explicit evaluation of $e_{\SG_n}(\F^\pm_{g,n})$ is to quantify the action of $\SG_n$ 
that permutes the legs
of the generators of $\F^\pm_{g,n}$. %

Let $\UAut(G,\Phi)$ be the set of automorphisms of a forested graph $(G,\Phi)$ with $n$ legs that are \emph{allowed} to permute the leg-labels.
For instance, for the graph in Figure~\ref{fig:fg143}, which has a trivial leg-label-preserving automorphism group,
the group $\UAut(G,\Phi)$ is generated by the automorphism that mirrors the graph vertically and permutes the 
leg-labels by $(12)(34)$.
We have a map $\pi_G: \UAut(G,\Phi) \rightarrow \SG_n$, given by only looking at the induced permutation on the leg-labels. The kernel of this map is equal to $\Aut(G,\Phi)$, the group of leg-label-fixing automorphisms of $(G,\Phi)$.

The $\SG_n$-action gives a representation $\rho: \SG_n \rightarrow \GL(C_k(\F^\pm_{g,n}))$. 
For a specific $\pi \in \SG_n$, the linear map $\rho(\pi)\in \GL(C_k(\F^\pm_{g,n}))$ replaces each generator of $C_k(\F^\pm_{g,n})$ with the respective generator where the leg labels are permuted by $\pi$.
The character of $C_k(\F^\pm_{g,n})$ is the composition of $\rho$ with the trace, $\chi_{C_k(\F^\pm_{g,n})} = \Tr \circ \rho$.
As $\rho(\pi)$ maps each generator to a multiple of another generator,
it is sufficient to look at generators that happen to be Eigenvectors of $\rho(\pi)$ to compute $\Tr(\rho(\pi))$.
Let $(G,\Phi,\sigma^\pm)$ be a connected $(\pm)$-orientable forested graph with a $(\pm)$-marking $\sigma^\pm$ corresponding to a generator of $C_k(\F^\pm_{g,n})$.
This generator is an Eigenvector of $\rho(\pi)$ if the forested graph $(G,\Phi)$ has a non-leg-label-fixing automorphism $\alpha \in \UAut(G,\Phi)$ such that $\pi_G(\alpha) = \pi$. 
The following lemma describes the Eigenvalue corresponding to such an Eigenvector. It is either $+1$ or $-1$.
\begin{lemma}
If, for given $\pi \in \SG_n$, the generator $(G,\Phi,\sigma^\pm) \in C_k(\F^\pm_{g,n})$ is an Eigenvector of $\rho(\pi)$, then
the corresponding Eigenvalue is $\xi^\pm(G,\Phi,\alpha)$,
where $\alpha$ is some representative $\alpha \in \pi_G^{-1}(\pi)$.
\end{lemma}
\begin{proof}
By Proposition~\ref{prop:genorient}, $(G,\Phi)$ cannot have a leg-label-fixing automorphism that flips the orientation. However, an automorphism from the larger group $\UAut(G,\Phi)$ can change the sign of the orientation
(i.e.~the sign of the ordering and basis given by the $(\pm)$-marking in Definition~\ref{def:fgc} can be flipped).
The map $\rho(\pi)$ does so if $\xi^\pm(G,\Phi,\alpha)=-1$ for a representative of the kernel $\alpha \in \pi_G^{-1}(\pi)$. As $\alpha \mapsto \xi^\pm(G,\Phi,\alpha)$ gives a group homomorphism $\UAut(G,\Phi) \rightarrow \Z/2\Z$ and as $\xi^\pm(G,\Phi,\alpha)=1$ for all $\alpha \in \ker\pi_G$, it does not matter which representative we pick.
\end{proof}

Summing over all such Eigenvalues of $\rho(\pi)$ gives the value of 
$\Tr(\rho(\pi)) = \chi_{C_k(\F^{\pm}_{g,n})}(\pi)$: %

\begin{corollary}
\label{cor:trace}
For each $\pi \in \SG_n$, the character of $C_k(\F^\pm_{g,n})$ is
\begin{align*} \chi_{C_k(\F^{\pm}_{g,n})}(\pi) = \sum \xi^\pm(G,\Phi,\alpha_{(G,\Phi,\pi)}), \end{align*}
where we sum over all isomorphism classes of $(\pm)$-orientable forested graphs $(G,\Phi)$ 
with $k=|\Phi|$ forest edges,
for which the preimage $\pi_G^{-1}(\pi) \subset \UAut(G,\Phi)$ is nonempty 
(i.e.~$(G,\Phi)$ has at least one non-leg-label-fixing automorphism that permutes the leg labels by $\pi$),
and $\alpha_{(G,\Phi,\pi)}$ is some representative of such an automorphism in $\pi_G^{-1}(\pi) \subset \UAut(G,\Phi)$.
\end{corollary}
To continue, it is convenient to pass to a sum over all connected forested graphs without restrictions on the orientability to get better combinatorial control over the expression:
\begin{corollary}
\label{cor:chain_trace}
For $g,n \geq 0$ with $2g - 2 + n \geq 0$ and $k \geq 0$, we have %
\begin{align*} \chi_{C_k(\F^{\pm}_{g,n})}(\pi) = \sum_{[G,\Phi]} \frac{1}{|\Aut(G,\Phi)|} \sum_{\alpha \in \pi_G^{-1}(\pi)} \xi^\pm(G,\Phi,\alpha), \end{align*}
where we sum over all isomorphism classes of (not necessarily orientable) connected forested graphs $[G,\Phi]$ of rank $g$, $n$ legs and $k=|\Phi|$ forest edges.
\end{corollary}
\begin{proof}
As $\alpha \mapsto \xi^\pm(G,\Phi,\alpha)$ also gives a map $\Aut(G,\Phi)\rightarrow \Z/2\Z$, 
we have by Definition~\ref{def:orient}
\begin{align*} \sum_{\alpha \in \ker \pi_G} \xi^\pm(G,\Phi,\alpha) = \begin{cases} |\Aut(G,\Phi)|&\text{ if $(G,\Phi)$ is $(\pm)$-orientable} \\
0 & \text{ else} \end{cases} \end{align*}
The statement follows, as $ \sum_{\alpha \in \pi_G^{-1}(\pi)} \xi^\pm(G,\Phi,\alpha) = \xi^\pm(G,\Phi,\alpha') \sum_{\alpha \in \ker \pi_G} \xi^\pm(G,\Phi,\alpha),$
for any representative $\alpha' \in \pi_G^{-1}(\pi)$ and the formula from Corollary~\ref{cor:trace} for $\chi_{C_k(\F^{\pm}_{g,n})}(\pi)$.
\end{proof}

With this we finally obtain our first explicit formula for $e_{\SG_n}(\F^\pm_{g,n})$:
\begin{theorem}
For $g,n \geq 0$ with $2g - 2 + n \geq 0$, %
\label{thm:equiv}
\begin{align*} e_{\SG_n}(\F^\pm_{g,n}) &= \sum_{[G,\Phi]_U} \frac{(-1)^{|\Phi|}}{|\UAut(G,\Phi)|} \sum_{\alpha \in \UAut(G,\Phi)} \xi^\pm(G,\Phi,\alpha) p^{\pi_G(\alpha)}, \end{align*}
where we sum over all isomorphism classes of connected forested graphs $[G,\Phi]_U$ of rank $g$ with $n$ \emph{unlabeled} legs.

\end{theorem}
\begin{proof}
We can plug the statement of Corollary~\ref{cor:chain_trace} into the definition of the Fr\"obenius characteristic in Proposition~\ref{prop:equivC} to get
\begin{align*} \ch(\chi_{C_k(\F^\pm_{g,n})}) = \frac{1}{n!} \sum_{\pi \in \SG_n} p^\pi \sum_{[G,\Phi]} \frac{1}{|\Aut(G,\Phi)|} \sum_{\alpha \in \pi_G^{-1}(\pi)} \xi^\pm(G,\Phi,\alpha). \end{align*}
Next, we can merge the first and the third sum into a summation over the whole group $\UAut(G,\Phi)$, 
as the preimages of different $\pi \in \SG_n$ partition $\UAut(G,\Phi)$:
\begin{align*} \ch(\chi_{C_k(\F^\pm_{g,n})}) = \frac{1}{n!} \sum_{[G,\Phi]} \frac{1}{|\Aut(G,\Phi)|} \sum_{\alpha \in \UAut(G,\Phi)} \xi^\pm(G,\Phi,\alpha) p^{\pi_G(\alpha)} . \end{align*}
Let $L_S(G,\Phi)$ be the set of isomorphism classes of forested graphs $(G,\Phi)$ that only differ by a relabeling of the legs.
Obviously, $\SG_n$ acts transitively on $L_S(G,\Phi)$ by permuting the leg-labels.  The stabilizer of this action is the image of $\pi_G:\UAut(G,\Phi)\rightarrow \SG_n$.
By the orbit-stabilizer theorem, $n! = |\SG_n| = |\im \pi_G| |L_S(G,\Phi)|$.
By the short exact sequence, $1 \rightarrow \Aut(G,\Phi) \rightarrow \UAut(G,\Phi ) \rightarrow \im \pi_G \rightarrow 1$, we have $|\im \pi_G| = |\UAut(G,\Phi)|/ |\Aut(G,\Phi)|$. Hence,
\begin{align*} \ch(\chi_{C_k(\F^\pm_{g,n})}) = \sum_{[G,\Phi]} \frac{1}{|L_S(G,\Phi)| |\UAut(G,\Phi)|} \sum_{\alpha \in \UAut(G,\Phi)} \xi^\pm(G,\Phi,\alpha) p^{\pi_G(\alpha)} . \end{align*}
The terms in this sum do not depend on the leg-labeling of the graphs, so we can just sum over non-leg-labeled graphs and remove the $|L_S(G,\Phi)|$ in the denominator.
\end{proof}
\subsection{Disconnected forested graphs}

The expression for $e_{\SG_n}(\F^\pm_{g,n})$ in Theorem~\ref{thm:equiv} involves a sum over \emph{connected} forested graphs without leg-labels. To eventually get an effective generating function for $e_{\SG_n}(\F^\pm_{g,n})$, it is convenient to pass to an analogous formula that sums over \emph{disconnected} forested graphs. 
Moreover, it is helpful to change from grading the graphs by rank to grading them by their negative Euler characteristic. 
For connected graphs, this is just a trivial shift as 
the negative Euler characteristic of such graphs fulfills $\rank H_1(G,\Z) - \rank H_0(G,\Z) = g-1$.
So, we define for all $t \in \Z$ and $n \geq 0$,
\begin{align} \label{eq:equiv_disc} \ee_{t,n}^\pm = \sum_{[G,\Phi]_U} \frac{1}{|\UAut(G,\Phi)|} \sum_{\alpha \in \UAut(G,\Phi)} \xi^\pm(G,\Phi,\alpha) p^{\pi_G(\alpha)}, \end{align}
where we sum over all isomorphism classes of (possibly disconnected) forested graphs $[G,\Phi]_U$ of Euler characteristic $\rank H_0(G,\Z) - \rank H_1(G,\Z) = -t$ and $n$ unlabeled legs.
There is only a finite number of such graphs, so the sum is finite and $ \ee_{t,n}^\pm $
is a symmetric polynomial in $\Lambda_n$.

Let $\widehat \Lambda$ be the \emph{ring of formal symmetric power series}
$\widehat \Lambda = \lim_{\leftarrow n} \Q[[x_1,\ldots,x_n]]^{\SG_n}$.
Later in this section we will prove a generating function,
expressed as a power series with coefficients in $\widehat \Lambda$,
 for the polynomials $\ee_{t,n}^\pm$. 
Before that, we explain how we can translate them into our desired $\SG_n$-equivariant Euler characteristics:
\begin{proposition}
\label{prop:exp}
We define the following Laurent series over the ring of formal symmetric power series 
\begin{align*} \mathbf e^\pm(\hbar, \bb p) &= \sum_{\substack{ g,n \geq 0\\ 2g-2+n \geq 0}} e_{\SG_n}(\F^\pm_{g,n}) (\pm \hbar)^{g-1}, \\
\mathbf E^\pm(\hbar, \bb p) &= \sum_{t\in \Z} \sum_{n \geq 0} \ee_{t,n}^\pm (\pm \hbar)^t, \end{align*}
which are elements of $\widehat{ \Lambda}((\hbar))$.
Both are related by the \emph{plethystic exponential}
\begin{align*} \mathbf E^\pm(\hbar, \bb p) = \exp\left( \sum_{k \geq 1} \frac{\mathbf e^\pm(\hbar^k, \bb p_{[k]})}{k} \right), \end{align*}
where $\mathbf e^\pm(\hbar^k, \bb p_{[k]})$ denotes the power series $\mathbf e^\pm$ with the substitutions $\hbar \rightarrow \hbar^k$ and $p_{i} \rightarrow p_{ik}$.
\end{proposition}
\begin{proof} 
The combinatorial argument for \cite[Proposition~3.2]{BV2} applies to
translate between the sum over \emph{connected} forested graphs in Theorem~\ref{thm:equiv}
to a sum over disconnected forested graphs in eq.~\eqref{eq:equiv_disc}.
The strategy goes back to P\'olya \cite{polya1937} (see also \cite[Chapter~4.3]{BLL}). 
Briefly, each summand 
${\mathbf e^\pm(\hbar^k, \bb p_{[k]})/k}$ in
the exponent of the stated formula 
counts pairs consisting of a $k$-tuple of mutually isomorphic forested graphs
and an automorphism that cyclically permutes the different graphs.
Such automorphisms give rise to different sign factors depending on the orientation of the graphs. Accounting for these signs gives the minus signs in front of $\hbar$ in the $(-)$-orientation case (see~\cite[Theorem~5.1]{BV2}). 
\end{proof}
We can solve the equation in the statement above for the generating functions $\mathbf e^\pm(\hbar, \bb p)$ 
and therefore obtain the value of $e_{\SG_n}(\F^\pm_{g,n})$ if we know sufficiently many coefficients of 
$ \mathbf E^\pm(\hbar, \bb p) $.
The required inverse transformation of power series is the plethystic logarithm:
\begin{corollary}
\label{cor:log}
\begin{align*} \mathbf e^\pm(\hbar, \bb p) = \sum_{k \geq 1} \frac{\mu(k)}{k} \log \mathbf E^\pm(\hbar^k, \bb p_{[k]}), \end{align*}
where $\mu(k)$ is the number-theoretic M\"obius function.
\end{corollary}
\begin{proof}
Let $g_1,g_2,\ldots \in \widehat \Lambda((\hbar))$ such that  
$f_n = \sum_{k\geq 1} {g_{nk}}/{(nk)}$ converges uniformly in $n$.
It follows from the definition of the M\"obius function,
$\sum_{d|n} \mu(d) = 0$ for all $n \geq 2$ and $\mu(1)=1$,
that 
$g_1 = \sum_{k \geq 1} {\mu(k)} f_{k}/k$.
With 
$ g_n = \mathbf e^\pm(\hbar^n, \bb p_{[n]}) $
and $f_n = \log \mathbf E^\pm(\hbar^n, \bb p_{[n]})$
, the statement now follows form Proposition~\ref{prop:exp}.
\end{proof}

\subsection{Generating functions}
\label{sec:genfun}
In this section we will give the desired generating function for the 
polynomials $\ee_{t,n}^\pm$. Together with the discussion in the last section,
this generating function will give us an effective formula for the 
$\SG_n$-equivariant Euler characteristic $e_{\SG_n}(\F^\pm_{g,n})$.

The generating function is closely related to the one given in \cite[Theorem~3.12]{BV2} and we refer to the argument given there.
As in \cite{BV2}, we define the following power series in $\bb q = q_1,q_2,\ldots$
 \begin{align} \label{eq:V} {\mathbf V}(\bb q)&=q_1+\frac{q_1^2}{2}-\frac{q_2}{2}-(1+q_1)\sum_{k \geq 1} \frac{\mu(k)}{k} \log(1+q_k). \end{align}
See \cite{BV2} for the first coefficients of $\mathbf V(\bb q)$.
Moreover, we define two power series $\mathbf F^+$ and $\mathbf F^-$ 
both in two infinite sets of variables
$\bb q = q_1,q_2, \ldots$,
 $\bb p = p_1,p_2,\ldots$ and a single variable $u$, 
\begin{align} \label{eq:F} {\mathbf F^\pm}(u,\bb q,\bb p) =\exp \left(\sum_{k \geq 1} (\pm1)^{k+1} u^{-2k} \frac{{\mathbf V}( (u \cdot \bb q)_{[k]}) + u^k q_k p_k}{k}\right), \end{align}
where ${\mathbf V}((u \cdot \bb q)_{[k]})$ means that we replace each variable $q_i$ in ${\mathbf V}(\bb q)$ with $u^{ki} q_{ki}$.

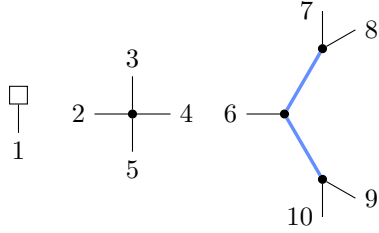
\begin{figure}
    \centering
    \begin{tikzpicture} \foreach \angle in {60,300} { \draw (\angle:1) node[vertex] (v\angle) {}; } \draw[col2, line width=1.3pt] (v) -- (v60); \draw[col2, line width=1.3pt] (v) -- (v300); \draw (v60) -- ++(90:.5) node[left] {$7$}; \draw (v60) -- ++(30:.5) node[right] {$8$}; \draw (v300) -- ++(330:.5) node[right] {$9$}; \draw (v300) -- ++(270:.5) node[left] {$10$}; \node[vertex] at (0,0) (v) {}; \draw (v) -- ++(-.5,0) node[left] {$6$}; \node[vertex] at (-2,0) (w) {}; \draw (w) -- ++(-.5,0) node[left] {$2$}; \draw (w) -- ++(0,.5) node[above] {$3$}; \draw (w) -- ++(.5,0) node[right] {$4$}; \draw (w) -- ++(0,-.5) node[below] {$5$}; \node[draw,rectangle] at (-3.5,.25) (u) {}; \draw (u) -- ++(0,-.5) node[below] {$1$}; \end{tikzpicture}
    \caption{Extended forest with one special component}
    \label{fig:extforest}
\end{figure}

Recall that a forest is an acyclic graph. As graphs, also forests are required to be admissible that means they have no vertices of degree $0$ and $2$.
Again, the univalent vertices are interpreted as the legs of the forest.
An \emph{extended} forest is a forests that is additionally allowed to have \emph{special} univalent vertices that are always connected by a 1-cell to a leg.
In contrast to graphs, we do not allow extended forests to have components that consist of 
two adjacent legs or two adjacent special vertices.
See also \cite[Section~3.7]{BV2} for a discussion of extended forests.
The legs of a forest are labeled by integers $\{1,\ldots,s\}$. The special vertices remain unlabeled.
Figure~\ref{fig:extforest} depicts an extended forest with three components; one of which is special. The special vertex is marked as a box. Legs are again drawn as labeled half-edges.

We write $k(\Phi)$, $s(\Phi)$ and $n(\Phi)$ for the 
total number of connected components, the number of legs and the number of special vertices of an extended forest $\Phi$.
An automorphism $\gamma \in \UAut(\Phi)$ that is allowed to permute the leg-labels gives rise to a permutation of the edges of $\Phi$.
We write $e_\gamma(\Phi)$ for the numbers of orbits of this permutation.
An automorphism $\gamma \in \UAut(\Phi)$ also gives rise to a permutation $\pi_G(\gamma) \in \SG_{s(\Phi)}$ of the legs, a permutation $\pi_G^\star(\gamma)\in \SG_{n(\Phi)}$ of the special vertices and a permutation $\gamma_{H_0(\Phi)}\in \SG_{h_0(\Phi)}$ of the connected components of $\Phi$.

\begin{proposition}
\label{prop:genFpm}
The generating functions ${\mathbf F^\pm}$ count signed pairs of an extended forest $\Phi$ and an automorphism $\gamma$ of $\Phi$. 
Explicitly, 
\begin{align*} {\mathbf F^+}(u,\bb q,\bb p) &= \sum_{(\Phi,\gamma)} \phantom{ \sign(\gamma_{H_0(\Phi)}) } (-1)^{e_\gamma(\Phi)} u^{s(\Phi)-2k(\Phi)} p^{\pi_G^\star(\gamma)} \frac{q^{\pi_G(\gamma)} }{s(\Phi)!} ,\\
{\mathbf F^-}(u,\bb q,\bb p) &= \sum_{(\Phi,\gamma)} \sign(\gamma_{H_0(\Phi)}) (-1)^{e_\gamma(\Phi)} u^{s(\Phi)-2k(\Phi)} p^{\pi_G^\star(\gamma)} \frac{q^{\pi_G(\gamma)} }{s(\Phi)!} , \end{align*}
where we sum over all pairs of an extended forest $\Phi$ and 
all $\gamma \in \UAut(\Phi)$.
\end{proposition}
\begin{proof}
This statement is a slight generalization of \cite[Proposition~3.10]{BV2} and \cite[Proposition~5.5]{BV2}. 
Here, we also allow forests to have special vertices and modify the generating function accordingly.
The term 
$(\pm1)^{k+1} u^{-k} q_k p_k/k$ 
in eq.~\eqref{eq:F} accounts for these special components.
Explicitly, it stands for $k$ special components that are cyclically permuted by the overall automorphism that is acting on the forested graph
(see, \cite[Lemma~3.3]{BV2} or \cite{BLL}). 
Each component contributes two negative powers of $u$, because it adds one component, and one positive power of $u$ as 
it adds one leg. So, we mark a cycle of $k$ special components with $u^{-k}$.
The different signs for the odd case are a consequence of \cite[Lemma~5.3]{BV2} and the fact that each special vertex counts as a new connected component of $\Phi$.
\end{proof}

We will use the coefficient extraction operator notation. That means, we denote the coefficient of a power series $f(\bb q)$ in front of $q^\lambda$ as $[q^\lambda] f(\bb q)$.
\begin{corollary}
\label{cor:coeffs}
The coefficient
$[u^{2t} q^\mu p^\lambda] \mathbf F^\pm(u,\bb q, \bb p)$ vanishes if 
$|\mu| = \sum_i \mu_i > 6t + 4 |\lambda|$.
\end{corollary}
\begin{proof}
A \emph{maximal} $(r,n)$-forest is an extended forest that consists of $r$ copies of a degree $3$ vertex with three legs, and $n$ special components 
of a special vertex and one leg.
All extended forests with $n$ special components can be obtained by first starting with a maximal $(r,n)$-forest, subsequently 
creating new edges by gluing together pairs of legs of different components and finally by contracting edges.
The difference $s(\Phi)-2k(\Phi)$ is left invariant by these gluing and contracting operations, but the number of legs $s(\Phi)$ decreases 
each time we glue together a pair of legs.
It follows that maximal $(r,n)$-forests have the maximal number of legs for fixed $n$.
Such a forest contributes a power $u^{3r-2r+n-2n} q^\mu p^\lambda= u^{r-n} q^\mu p^\lambda$ to the generating function,
for some partitions $\mu,\lambda$ with $|\mu| = 3r+n$ and $|\lambda| = n$.
Fixing $2t = r-n$ gives $|\mu| = 6t + 4n$. 

Alternatively, the statement can also be verified by expanding $\mathbf V$ and $\mathbf F^\pm$ from eqs.~\eqref{eq:V}--\eqref{eq:F}.
\end{proof}

We define two sets of numbers, $\eta_\lambda^+$ and $\eta^-_\lambda$, that are indexed by an integer partition, $\lambda \parts s$. These numbers combine the definitions in Corollaries 3.5 and 5.8 of \cite{BV2}, where $\eta^+_\lambda$ is denoted as $\eta_\lambda$ and $\eta^-_\lambda$ as $\eta^{\mathrm{odd}}_\lambda$.
The discussion around these corollaries also includes a detailed combinatorial interpretation of these numbers: they count (signed) fixed-point free involutions that commute with a given permutation of cycle type $\lambda$.
An alternative notation for an integer partitions is 
$\lambda = [1^{m_1} 2^{m_2} \cdots ]$, where $m_k$ denotes 
the number of parts of size $k$ in $\lambda$.
Let
$$
 \eta^\pm_{\lambda}  = \prod_{k = 1}^s \eta^\pm_{k, m_k}, \text{
 where }
\eta_{k,\ell}^\pm=
\begin{cases}
0 & \text{if $k$ and $\ell$ are odd}\\
k^{\ell/2} (\ell -1)!!  & \text{if $k$ is odd and $\ell$ is even}\\
\sum_{r=0}^{\lfloor \ell/2 \rfloor} (\pm 1)^{\ell k/2+r} \binom{\ell}{2r} k^r (2r-1)!! & \text{else}
\end{cases}
$$

With this we get an effective expression for the polynomials $\ee_{t,n}^\pm$ from eq.~\eqref{eq:equiv_disc}:
\begin{theorem}
\label{thm:ee}
\begin{align*} \ee_{t,n}^\pm = \sum_\mu \eta^\pm_{\mu} \sum_{\lambda \parts n} p^\lambda [u^{2t} q^\mu p^\lambda] \mathbf F^\pm(u,\bb q, \bb p), \end{align*}
where we sum over all integer partitions $\mu$ and all partitions $\lambda$ of $n$.
The terms in the sum are only non-zero for a finite number of such partitions.
Moreover, all terms vanish if $t < - \frac23 n$.
\end{theorem}

\begin{proof}
By Corollary~\ref{cor:coeffs}, the sum over $\mu$ has finite support 
and all terms vanish if $t < - \frac23 n$.

Matching all legs of an extended forest in pairs gives a graph with a marked forest and special univalent vertices. 
We get this graph by gluing together the legs as described by a matching. 
Gluing together two legs creates a new 1-cell between the vertices that are adjacent to the legs and forgets about the legs and the 1-cells that are incident to them. 
The special vertices are subsequently promoted to (unlabeled)
legs of the resulting forested graph.  
All forested graphs with unlabeled legs can be obtained this way.

As an example consider the extended forest in Figure~\ref{fig:extforest}. 
We can match the legs in the pairs $(1,2),(3,7),(4,6),(5,10)$ and $(8,9)$.
Gluing together these legs and promoting the special vertex to a new leg recreates the forested graph in Figure~\ref{fig:fg312} without the leg-label.

In general, if the extended forest that we start with has $k$ connected components and $s$ legs, then 
the graph that we obtain after the gluing has Euler characteristic $k-s/2$.
By Proposition~\ref{prop:genFpm} and the definition of the polynomials $\ee_{t,n}$ in eq.~\eqref{eq:equiv_disc} we therefore extract the correct coefficient as all forests $\Phi$ with $s(\Phi) - 2k(\Phi) = 2t$ contribute to the coefficient of $u^{2t}$.
See the proof of \cite[Proposition~3.10]{BV2} and its odd version \cite[Proposition~5.5]{BV2}
for the discussions of automorphisms and sign factors which apply also in our generalized case.
\end{proof}

Together Corollary~\ref{cor:MGF}, Corollary~\ref{cor:log} and Theorem~\ref{thm:ee} give an effective algorithm for the computation of the polynomials $e_{\SG_n}(\MG_{g,n})$ and $e_{\SG_n}^\mathrm{odd}(\MG_{g,n})$. In summary:

\begin{theorem}
\label{thm:comp}
Fix $\chi >1$.
To compute $e(\F^\pm_{g,n})$ for all $2g-2+n > 0$ and $g+1+n \leq \chi$,
\begin{enumerate}
\item Compute the coefficients of $\mathbf{V}$ up to homogeneous order $6\chi$ in the $\bb q$-variables by expanding the power series defined in eq.~\eqref{eq:V}.
\item Compute the coefficients of $\mathbf{F}^\pm$ up to homogeneous order $\chi$ in $u^2$ and in the $\bb p$ variables by expanding the power series defined in eq.~\eqref{eq:F}.
\item Compute the polynomials $\ee_{t,n}^\pm$ for all pairs $(t,n)$ with $n \geq 0$, $t + n \leq \chi$ and $t \geq -\frac23 n$ using Theorem~\ref{thm:ee}.
\item Compute the polynomials $e_{\SG_n}(\F^\pm_{g,n})$ using the formula from Corollary~\ref{cor:log}.
\end{enumerate}
\end{theorem}

The formulas in Corollary~\ref{cor:MGF} and Proposition~\ref{prop:MGF-special} can be used to translate the result into the numbers $e(\MG_{g,n}),e^\mathrm{odd}(\MG_{g,n}),e(\MG_{g,n}^{\SG_n})$ and $e^\mathrm{odd}(\MG_{g,n}^{\SG_n})$.
To compute the integers $\sum_k (-1)^k c^k_{g,\lambda}$ that are the alternating sums over the multiplicities of the irreducible representations in the respective cohomology of $\MG_{g,n}$, we have to write the polynomials $e_{\SG_n}(\MG_{g,n})$ and $e_{\SG_n}^\mathrm{odd}(\MG_{g,n})$ in terms of Schur polynomials as in eq.~\eqref{eq:Sndef}. We can do so using the Murnaghan-Nakayama rule.

\subsection{Implementation of Theorem~\ref{thm:comp} in \texorpdfstring{\texttt{FORM}}{FORM}}
\label{sec:form}

The most demanding computational step in Theorem~\ref{thm:comp} is the expansion of the power series $\mathbf F^\pm$ as defined in eq.~\eqref{eq:F} in the formal variables $u,\bb p$ and $\bb q$. Conventional computer algebra struggles with such expansions; usually only a small number of terms are accessible. To be able to apply Theorem~\ref{thm:comp} at moderately large values of $\chi$, we use the \texttt{FORM} programming language. \texttt{FORM} is designed to deal with large analytic expressions that come up in high-energy physics.

A \texttt{FORM} program that implements Theorem~\ref{thm:comp} is included in the ancillary files to this article in the file \texttt{eMGgn.frm}. It can be run with the command \texttt{form eMGgn.frm} after downloading and installing \texttt{FORM} from \url{https://github.com/vermaseren/form.git}. The syntax and details of the code are described in a \texttt{FORM} tutorial \cite{FORMcourse}. We used  \texttt{FORM} version \emph{5 beta} for our computations.

The output of the program is a power series in which each coefficient is a symmetric function that describes the respective $\SG_n$-equivariant Euler characteristic. These coefficients are given in the power sum basis of the ring of symmetric functions.
To translate the output into Schur symmetric function via the Murnaghan-Nakayama rule, we used  \texttt{Sage}~\cite{sagemath}. 

\begin{remark}
By employing the \emph{Feynman transform} introduced by Getzler and Kapranov \cite{getzler1998modular}, modified versions of this program can effectively compute the Euler characteristics of \emph{modular operads}. Furthermore, by combining Joyal's theory of species (see, e.g., \cite{BLL}) with findings from the first author's thesis \cite{Borinsky:2018mdl}, the program can be adapted to count various combinatorial objects, such the number of isomorphism classes of admissible graphs of rank $n$ \cite{OEIS}.
\end{remark}

\subsection{Large \texorpdfstring{$g$}{g} asymptotics of the Euler characteristics of \texorpdfstring{$\MG_{g,n}$}{MGgn}}
\label{sec:asy}
The \emph{virtual} or \emph{rational} Euler characteristic $\chi(G)$ is an 
invariant of a group $G$ that is often better behaved than the usual Euler characteristic. %

Recall that the notation `$f(g) \sim h(g)$ for large $g$' means that $\lim_{g\rightarrow \infty} f(g)/h(g) = 1$.
By the short exact sequence 
$1 \rightarrow F_g^n \rightarrow \Gamma_{g,n} \rightarrow \Out(F_{g}) \rightarrow 1$ \cite{conant2016assembling}, we have $\chi(\Gamma_{g,n}) = (1-g)^n \chi(\Out(F_{g}))$.
The numbers $\chi(\Out(F_{g}))$ can be computed using \cite[Proposition 8.5]{BV} and 
the asymptotic growth rate of $\chi(\Out(F_{g}))$ is known explicitly by \cite[Theorem~A]{BV}. Using this together with the formula for $\chi(\Gamma_{g,n})$ and Stirling's approximation, we find that for fixed $n \geq 0$,
\begin{align*} \chi(\Gamma_{g,n}) \sim (-1)^{n+1} g^{n} \left( \frac{g}{e}\right)^g / (g\log g)^2 \text{ for large } g.  \end{align*}

In \cite{BV2}, it was proven that 
$e(\MG_{g,0}) \sim e^{-\frac14} \chi(\Out(F_{g}))$ 
and 
$e^\mathrm{odd}(\MG_{g,0}) \sim e^{\frac14} \chi(\Out(F_{g}))$ for large $g$.
It would be interesting to make similar statements about $\MG_{g,n}$ for $n \geq 1$.
Using our data on the Euler characteristics of $\MG_{g,n}$, we empirically verified the following conjecture which generalizes the known asymptotic behaviour of $\MG_{g,0}$ to $\MG_{g,n}$ for all $n \geq 0$.
\begin{conjecture} 
\label{conj:asymp}
For fixed $n \geq 0$, we have for large $g$
\begin{align*} e(\MG_{g,n}) &\sim e^{-\frac14} \chi(\Gamma_{g,n}) & e^\mathrm{odd}(\MG_{g,n}) &\sim e^{\frac14} \chi(\Gamma_{g,n}) \\
e(\MG_{g,n}^{\SG_n}) &\sim e^{-\frac14} \chi(\Gamma_{g,n})/n! & e^\mathrm{odd}(\MG_{g,n}^{\SG_n}) &\sim e^{\frac14} \chi(\Gamma_{g,n})/n! \end{align*}
\end{conjecture}
Proving this conjecture should be feasible by generalizing the analytic argument in \cite[Sec.~4]{BV2}.

\section{The large-$n$ \texorpdfstring{$\SG_n$}{Sn}-invariant cohomological stability of \texorpdfstring{$\MG_{g,n}$}{MGgn}}
\label{sec:stable}
Our data in the Tables~\ref{tab:eeven}--\ref{tab:eEodd} exhibit an obvious pattern:
For fixed $g$  the \emph{$\SG_n$-invariant} Euler characteristics $e(\MG_{g,n}^{\SG_n})$ and $e^\mathrm{odd}(\MG_{g,n}^{\SG_n})$ appear to be constant for all $n\geq g$. %
Unfortunately, this is not manifest from our formulas (i.e.~from Theorem~\ref{thm:comp}).
To explain this pattern, we will prove the stabilization of the associated cohomologies.
\begin{theorem}
\label{thm:stab}
Fix $\mathbb Q_\rho \in \{ \mathbb Q, \widetilde{ \mathbb Q} \}$.
The cohomology $H^\bullet(\MG_{g,n};\Q_\rho)^{\SG_n}$ stabilizes for $n\rightarrow \infty$:
If  $n \geq g \geq 2$, then there are  isomorphisms
$ H^k(\MG_{g,n};\Q_\rho)^{\SG_n}\rightarrow H^k(\MG_{g,n+1};\Q_\rho)^{\SG_n}$ for all $k$.
\end{theorem}
This statement is a refinement of the known 
\emph{representational stability} \cite{church2013representation} of $\MG_{g,n}$, which was shown to hold in \cite{conant2016assembling}. 
In contrast to those previous results, Theorem~\ref{thm:stab} holds independently of the cohomological degree $k$.
We will prove this theorem below in Section~\ref{sec:lerayserre}
using an argument that is based on the  Lyndon--Hochschild--Serre spectral sequence and 
closely related to lines of thought in \cite{conant2016assembling} and \cite{saied2015fi}.

Unfortunately, our proof of Theorem~\ref{thm:stab} gives neither a concrete description of the stable cohomologies
$ H^\bullet(\MG_{g,\infty};\Q)^{\SG_\infty} $
and 
$ H^\bullet(\MG_{g,\infty};\tQ)^{\SG_\infty} $
nor gives explicit stabilization maps. A candidate for such a map is a generalization of the injection $H^\bullet(\Out(F_g);\Q) \rightarrow H^\bullet(\Aut(F_g);\Q)$ from Theorem~1.4 of \cite{conant2016assembling} to arbitrarily many legs.
The first few terms of the Euler characteristics associated to these stable cohomologies are tabulated in Table~\ref{tab:stable}.
The values are remarkably small in comparison to the value of the Euler characteristic of $\Out(F_g)$ for the respective rank (see the top and bottom row of Table~\ref{tab:eEeven} for a direct comparison). Empirically, the Euler characteristics of the stable cohomologies appear to grow exponentially and not super-exponentially.
Our argument for Theorem~\ref{thm:stab} does not give any hint why these Euler characteristics  are so small.
It would also be interesting to prove Theorem~\ref{thm:stab} using graph cohomology methods as this kind of stabilization appears to be a distinguished feature of Lie- and forest graph cohomology. For instance, commutative graph cohomology with legs does not stabilize in the strong sense observed here \cite{turchin2017commutative}.

\begin{table}
\renewcommand{\arraystretch}{1.5}
\begin{centering}
\begin{tabular}{c|r r r r r r r r r r r r r}
$g$& $1$&$2$&$3$&$4$&$5$&$6$&$7$&$8$&$9$&$10$&$11$&$12$&$13$ \\
\hline
$e(\MG_{g,\infty}^{\SG_\infty})$& $1$&$1$&$2$&$1$&$2$&$3$&$11$&$0$&$18$&$-7$&$71$&$-102$&$295$ \\
$e^{\mathrm{odd}}(\MG_{g,\infty}^{\SG_\infty})$& $0$&$1$&$-2$&$1$&$-2$&$3$&$-11$&$0$&$-18$&$-7$&$-71$&$-102$&$-295$ \\
\end{tabular}
\end{centering}
\caption{Euler characteristics for $H^\bullet(\MG_{n,\infty}; \Q)^{\SG_\infty}$ and $H^\bullet(\MG_{n,\infty}; \tQ)^{\SG_\infty}.$}
\label{tab:stable}
\end{table}

\subsection{Analogy to Artin's braid group}
The observed large-$n$ stabilization carries similarities with the cohomology of Artin's \emph{braid group} $B_n$ of 
equivalence classes of $n$-braids.

For a graph $G$ of rank $g$ with $n$ legs, let $\A_{g,n}$ be the group of homotopy classes of self-homotopy equivalences of $G$ that only fix the legs as a set and not point-wise, i.e.~$\A_{g,n}=\pi_0(\mathrm{HE}(G,\partial G)).$
By looking only at the action of $\A_{g,n}$ on the leg-labels $\{1,\ldots,n\}$, we get a surjective map to the symmetric group $\SG_n$, i.e.
\begin{align*} 1 \rightarrow \PA_{g,n} \rightarrow \A_{g,n} \rightarrow \SG_n \rightarrow 1, \end{align*}
where the `pure' group, $\PA_{g,n}$, is the subgroup of $\A_{g,n}$ that fixes the legs point-wise as defined in the introduction.
Analogously, the braid group $B_n$ maps surjectively to $\SG_n$. So,
$$
1\rightarrow P_n \rightarrow B_n \rightarrow \SG_n \rightarrow 1,
$$
where the kernel $P_n$ is the \emph{pure braid group}.
Whereas the braid groups $B_n$ exhibit 
\emph{homological stability},
i.e.~if $n\geq 3$, then $H_k(B_n;\Q) \cong H_k(B_{n+1};\Q)$ for all $k$ \cite{arnol1969cohomology}, the pure braid groups $P_n$ only satisfy representational stability \cite{church2013representation}. 
For instance, $H^\bullet(P_n;\Q)$ is an exterior algebra on $\binom{n}{2}$ generators modulo a $3$-term relation~\cite{arnol1969cohomology} and $H^\bullet(\PA_{1,n};\Q)$ is the even degree part of an exterior algebra on $n-1$ generators~\cite{conant2016assembling}. The numbers of generators of both algebras increase with~$n$.
The cohomology of $\A_{g,n}$ is equal to the 
$\SG_n$-invariant cohomology of $\MG_{g,n}$. So, analogously to the stabilization of $H^\bullet(B_n)$, Theorem~\ref{thm:stab} implies the large-$n$ cohomological stability of $\A_{g,n}$.

\subsection{The Lyndon--Hochschild--Serre spectral sequence}
\label{sec:lerayserre}

Our proof of Theorem~\ref{thm:stab} 
makes heavy use of tools from \cite{conant2016assembling}.
Following~\cite{conant2016assembling}, we abbreviate $\mathrm H = H^1(F_g)$ and think of it as an $\Out(F_g)$-module. The action of $\Out(F_g)$ on $\mathrm H$ factors through the action of $\GL_g(\Z)$ on $\mathrm H$. The \emph{q-th exterior power} of $\mathrm H$ is denoted as $\bigwedge \nolimits^q \mathrm H$. As $\rank H = g$, the determinant representation is recovered by $\bigwedge \nolimits^g \mathrm H = \tQ$. As before, we only work with modules over rational coefficients.

\begin{proposition}
\label{prop:spectral}
Fix $g,n$ with $g \geq 2$, $n \geq 0$.
There are two Lyndon--Hochschild--Serre spectral sequences with $E_2$ pages given by
\begin{align*} E_2^{p,q} &= H^p\left(\Out(F_g); \bigwedge\nolimits^q \mathrm H\right) \\
\widetilde {E}_2^{p,q} &= H^p\left(\Out(F_g); \tQ \otimes \bigwedge\nolimits^q \mathrm H\right) \end{align*}
for $p,q \geq 0$ with $p \leq 2g-3$ and $q \leq \min(g,n)$,
and $E_2^{p,q} = \widetilde {E}_2^{p,q} = 0$ for all other values of $p,q$.
Both spectral sequences converge:
$E_2^{p,q} \Rightarrow H^{p+q}(\PA_{g,n};\Q)^{\SG_n}$ and 
$\widetilde {E}_2^{p,q} \Rightarrow H^{p+q}(\PA_{g,n};\tQ)^{\SG_n}$.
\end{proposition}
The argument works along the lines of Section~3.2 of
\cite{conant2016assembling} followed by an application of Schur--Weyl duality and the projection to the $\SG_n$-invariant cohomology.
A similar argument can be found in \cite[Lemma~4.1]{saied2015fi} where it is used to prove representational stability of $H^\bullet(\PA_{g,n};\Q)$. 
\begin{proof}
For $g \geq 2$ and $n \geq 0$, the Lyndon--Hochschild--Serre spectral sequence associated to the group extension $1 \rightarrow F_g^n \rightarrow \PA_{g,n} \rightarrow \Out(F_g) \rightarrow 1$ 
is a first-quadrant spectral sequence with the second page given by 
$E_2^{p,q} = H^p(\Out(F_g); H^q(F_g^n;\Q)) \Rightarrow H^{p+q}(\PA_{g,n};\Q)$ (see~\cite[Section~3.2]{conant2016assembling}).
The virtual cohomological dimension of $\Out(F_n)$ is $2g-3$ \cite{CV},
so if $p > 2g-3$, then $E_2^{p,q} = \widetilde {E}_2^{p,q} = 0$.
By the Künneth formula, the $\SG_n$-module $H^q(F_g^n;\Q)$ 
vanishes if $q > n$ and for $0\leq q\leq n$ it 
is obtained by inducing the $\SG_{q}\times \SG_{n-q}$-module $\mathrm H^{\wedge q} \otimes V_{(n-q)}$ to $\SG_n$, (see~\cite[Lemma~3.4]{conant2016assembling} for details),
$$H^q(F_g^n;\Q) = 
\Ind^{\SG_n}_{\SG_q\times \SG_{n-q}} \left( H^{\wedge q} \otimes V_{(n-q)}\right),$$
where $V_{(n-q)}$ is the trivial representation of $\SG_{n-q}$ and $\mathrm H^{\wedge q}$ is $\mathrm H^{\otimes q}\otimes V_{(1^q)}$ with $V_{(1^q)}$ the alternating representation of $\SG_q$ and $\SG_q$ acts on $H^{\otimes q}$ by permuting the entries of the tensor product.

We can set up a similar spectral sequence with coefficients twisted by $\tQ$ with second page $\widetilde E_2^{p,q} = H^p(\Out(F_g); H^q(F_g^n;\tQ)) \Rightarrow H^{p+q}(\PA_{g,n};\tQ)$.
Note that $F_g^n$ acts trivially on $\tQ$. So, applying the Künneth formula to expand $H^q(F_g^n;\tQ)$ as an $\SG_n$-module gives 
$$H^q(F_g^n;\tQ) = \tQ \otimes 
\Ind^{\SG_n}_{\SG_q\times \SG_{n-q}} \left( H^{\wedge q} \otimes V_{(n-q)}\right),$$ which only differs by the $GL_g(\Z)$-action on $\tQ$. 
Schur--Weyl duality gives the irreducible decomposition of $\mathrm H^{\wedge q}$ as a module over $\GL(\mathrm H)\times \SG_q$ (see, e.g., \cite[Section~3.1]{conant2016assembling}):
\begin{align*} \mathrm H^{\wedge q} = \bigoplus_{\lambda \parts q} W_\lambda \otimes V_{\lambda'}, \end{align*}
where we sum over all partitions $\lambda$ of $q$ with at most $g$ rows  and $\lambda'$ is the transposed partition. Here,
$W_\lambda$ and $V_\lambda$ are the irreducible $\GL(\mathrm H)$- and $\SG_q$-representations associated to $\lambda$.

As $\SG_n$ only acts on the coefficients in $E_2^{p,q}$, we get
\begin{align*} E_2^{p,q} &= H^p\left(\Out(F_g); \Ind^{\SG_n}_{\SG_q\times \SG_{n-q}} \left( H^{\wedge q} \otimes V_{(n-q)}\right) \right) \\
 &= \bigoplus_{\lambda \parts q} H^p\left(\Out(F_g); W_\lambda \right) \otimes \Ind^{\SG_n}_{\SG_q\times \SG_{n-q}} \left( V_{\lambda'} \otimes V_{(n-q)}\right). \end{align*}
If we project to the trivial $\SG_n$-representation,
only the partition $\lambda = (1^q)$ contributes. Hence,
\begin{align*} (E_2^{p,q})^{\SG_n} &= H^p\left(\Out(F_g); W_{(1^q)} \right). \end{align*}
The representation $W_{(1^q)}$ is the $q$-th exterior power of the defining representation $\mathrm H$. It vanishes if $q > \rank H = g$.
The argument works analogously for $\widetilde E_2^{p,q}$.
\end{proof}

\begin{proof}[Proof of Theorem~\ref{thm:stab}]
For $n \geq g \geq 2$, the $E_2$ pages of the spectral sequences in 
Proposition~\ref{prop:spectral}
are independent of $n$.
Hence, by the spectral sequence comparison theorem, also $H^k(\PA_{g,n},\Q)^{\SG_n}$ and $H^k(\PA_{g,n},\tQ)^{\SG_n}$ must be independent of $n$.
\end{proof}

In the stable regime, the Euler characteristics $e(\MG_{g,n}^{\SG_n})$ and $e^\mathrm{odd}(\MG_{g,n}^{\SG_n})$ are equal up to a sign factor (see Table~\ref{tab:stable}). 
This can be explained as another  consequence of Proposition~\ref{prop:spectral},
and the following two technical lemmas.

For a graph $G$, the first cohomology $H^1(G) =H^1(G;\Q)$ is a representation of its automorphism group $\Aut(G)$. Let $H^1(G)^*$ be the corresponding dual representation.
\begin{lemma}
\label{lmm:dualrep}
$H^1(G) \cong H^1(G)^*$
as $\Aut(G)$-representations.
\end{lemma}
\begin{proof}
Each graph $G$ comes with an $\Aut(G)$-invariant 
symmetric positive definite bilinear form $H_1(G) \otimes H_1(G) \rightarrow \Q$, \emph{the graph Laplacian} (see, e.g.,~\cite[3.1.1]{baker2011jacobian}). Dually, we have an $\Aut(G)$-invariant pairing on $H^1(G)$ giving the stated isomorphism of $\Aut(G)$-representations.
\end{proof}

Let $\mathrm H^*$ be the $\Out(F_g)$-representation dual to $\mathrm H$
and $e(H^{\bullet}\left(G;M\right)) = \sum\limits_p (-1)^p \dim (H^{p}\left(G;M\right))$.
\begin{lemma}
\label{lmm:group_cohom_dual}
For $g \geq 2$ and $q \geq 0$,
$e\left( H^{\bullet}\left(\Out(F_g);\bigwedge\nolimits^q \mathrm H\right) \right)= e\left( H^{\bullet}\left(\Out(F_g);\bigwedge\nolimits^q \mathrm H^*\right) \right) $.
\end{lemma}
\begin{proof}
Using standard arguments that relate Culler--Vogtmann outer space to the rational group cohomology of $\Out(F_n)$, %
we are looking for a cochain model for the stated cohomologies.
Let $K_g$ be the \emph{spine of outer space},
i.e.~$K_g$ is a cubical complex in which each $k$-cube
corresponds to an isomorphism class of a forested graph $(G,\Phi,\mu)$ 
which is equipped with a \emph{marking}, an isomorphism $\mu:\pi_1(G) \rightarrow F_g$. Each cube is oriented by ordering the edges in $\Phi$.
The group $\Out(F_g)$ acts on $K_g$ by composition with the marking.
This action is not free, but it has finite stabilizers. 
Most importantly, $K_g$ is contractible \cite{CV}.
Let $M$ be $\bigwedge\nolimits^q \mathrm H$ or $\bigwedge\nolimits^q \mathrm H^*$, then by standard arguments
$H^{\bullet}\left(\Out(F_g);M\right) = H^{\bullet}\left( \mathrm{Hom}_{\Out(F_g)} (C_\bullet(K_g), M) \right)$
(see, e.g., \cite[Ch.~7]{Br}).
The action of $\Out(F_g)$ on the tuples $(G,\Phi,\mu)$ fixes the 
isomorphism class of the forested graph and acts transitively on the marking.
The stabilizer of this action is $\Aut(G,\Phi)$.
Hence,
\begin{align*} \mathrm{Hom}_{\Out(F_g)} (C_\bullet(K_g), M) = \mathrm{Hom} (C_\bullet(K_g), M)^{\Out(F_g)} = \bigoplus_{[G,\Phi]} \left( \sign_{(G,\Phi)} \otimes M \right)^{\Aut(G,\Phi)}, \end{align*}
where we take the sum over all isomorphism classes of 
forested graphs of rank $g$
or equivalently over all $\Out(F_g)$-orbits of 
isomorphism classes of marked forested graphs of rank $g$.
Here, $\sign_{(G,\Phi)}$ is the orientation module, the one-dimensional representation of $\Aut(G,\Phi)$ that is given by $\alpha \mapsto \sign(\alpha_\Phi)$ with $\alpha_\Phi$ the permutation that $\alpha \in \Aut(G,\Phi)$ induces on the forest edges. 
The group
$\Aut(G,\Phi)$ acts on $M$ via the map $\Aut(G,\Phi) \rightarrow \Out(F_g)$ that any
chosen marking representative $\pi_1(G)\rightarrow F_g$ induces.
As a $\Aut(G,\Phi)$-representation, $\mathrm H$ is hence 
equal to $H^1(G)$. Therefore, due to Lemma~\ref{lmm:dualrep},
we may replace the right-most $M$ in the displayed formula above with 
either $\bigwedge\nolimits^q H^1(G)$, or $\bigwedge\nolimits^q (H^1(G)^*) \cong (\bigwedge\nolimits^q H^1(G))^*$.
So, in the stated equality, both sides can be computed using cochain complexes with the same generators in each degree. The Euler characteristic only depends on the dimensions of the cochain spaces and not on the~differentials.
\end{proof}

\begin{proposition}
\label{prop:oddeven}
If $n \geq g \geq 2$, then $e(\MG_{g,n}^{\SG_n}) = (-1)^g e^\mathrm{odd}(\MG_{g,n}^{\SG_n})$.
\end{proposition}
\begin{proof}
There is an isomorphism of $\GL_{g}(\Z)$-representations: 
$\tQ \otimes \bigwedge\nolimits^{q} \mathrm H = \bigwedge\nolimits^{g} \mathrm H \otimes \bigwedge\nolimits^{q} \mathrm H \cong \bigwedge\nolimits^{g-q} \mathrm H^*$.
This can be seen by computing the characters of both 
$\bigwedge\nolimits^{g}\mathrm H \otimes \bigwedge\nolimits^{q} \mathrm H$ and $\bigwedge\nolimits^{g-q} \mathrm H^*$
and by observing that they differ by a multiple of the  character of $(\bigwedge\nolimits^{g})^{\otimes2}$,
which is the same as the trivial representation of $\GL_g(\Z)$, because $\det(h)^2 = 1$ for each $h\in \GL_g(\Z)$ (see, e.g., \cite[Ch.~7.A.2]{stanley1997enumerative2}).

Hence, for $n \geq g \geq 2$, we have for the $E_2$ pages of the spectral sequences in Proposition~\ref{prop:spectral}:
$$\widetilde E_2^{p,q} = H^{p}\left(\Out(F_g);\tQ \otimes \bigwedge\nolimits^q \mathrm H\right)
\cong 
H^{p}\left(\Out(F_g);\bigwedge\nolimits^{g-q} \mathrm H^*\right).
$$
Therefore, it follows from Lemma~\ref{lmm:group_cohom_dual} that $\sum_p (-1)^p \dim \widetilde E_2^{p,q} = \sum_p (-1)^p \dim E_2^{p,g-q}$.
So,
\begin{gather*} e(\MG_{g,n}^{\SG_n}) = \sum_{p,q} (-1)^{p+q} \dim E_2^{p,q} = \sum_{p,q} (-1)^{p+g-q} \dim E_2^{p,g-q} = (-1)^g e^\mathrm{odd}(\MG_{g,n}^{\SG_n}). \qedhere \end{gather*}
\end{proof}
Apart from the (up-to-sign) equality of the Euler characteristics, our argument does not tell us how even and odd $\SG_n$-invariant stable $n\rightarrow \infty$ cohomology of $\MG_{g,n}$ are related.

Moreover, the $E_2$ page of the spectral sequences in Proposition~\ref{prop:spectral} contains the whole (twisted) cohomology of $\Out(F_g)$ (e.g.~
 $E_2^{p,0} = H^p(\Out(F_g);\Q)$).
So, we also do not explain why the Euler characteristics of the large-$n$ stable  $\SG_n$-invariant cohomologies of $\MG_{g,n}$ are so small
in comparison to the Euler characteristic of $\Out(F_g)$.

\appendix
\section{Tables of Euler characteristics of \texorpdfstring{$\MG_{g,n}$}{MGgn}}
Larger versions of the following Tables~\ref{tab:eeven}-\ref{tab:eEschurodd} are included in the ancillary files for the arXiv version of this article:
Values  %
of the Euler characteristics
$e(\MG_{g,n})$,
$e^\mathrm{odd}(\MG_{g,n})$,
$e(\MG_{g,n}^{\SG_n})$ and 
$e^\mathrm{odd}(\MG_{g,n}^{\SG_n})$
for $g+n \leq 60$ can be found in \texttt{eMGgn.tsv}, \texttt{eMGgn-odd.tsv}, \texttt{eMGgn-modSn.tsv} and \texttt{eMGgn-modSn-odd.tsv}.
These larger tables are readable as plain text or with standard spreadsheet software.
The files \texttt{eOutFn.txt} and \texttt{eOutFn-odd.txt} list the values $e(\Out(F_g))$ and $e^\textrm{odd}(\Out(F_g))$ for all $g\leq 100$.
The full $\SG_n$-equivariant Euler characteristic is listed in the files \texttt{eMGgn-equiv.tsv} and \texttt{eMGgn-equiv-odd.tsv} that contain tables of the polynomials $e_{\SG_n}(\MG_{g,n})$ and $e_{\SG_n}^\mathrm{odd}(\MG_{g,n})$ for $g+n \leq 30$ respectively.

\newgeometry{hmargin=1cm,vmargin=2cm}
\begin{landscape}

\begin{table}
\begin{tabular}{c|r r r r r r r r r}
\backslashbox{$n$}{$g$}& 1&2&3&4&5&6&7&8&9 \\
\hline
$0$& &$1$&$1$&$2$&$1$&$2$&$1$&$1$&$-21$ \\
$1$& $1$&$1$&$1$&$2$&$0$&$3$&$3$&$31$&$154$ \\
$2$& $1$&$1$&$2$&$3$&$-2$&$1$&$-28$&$-153$&$-1486$ \\
$3$& $2$&$1$&$6$&$6$&$0$&$36$&$152$&$1423$&$12072$ \\
$4$& $4$&$0$&$18$&$4$&$-24$&$-86$&$-1062$&$-9474$&$-103392$ \\
$5$& $8$&$-4$&$61$&$-3$&$64$&$675$&$6421$&$72249$&$845821$ \\
$6$& $16$&$-19$&$202$&$-158$&$-69$&$-3453$&$-38028$&$-506827$&$-6971380$ \\
$7$& $32$&$-69$&$701$&$-831$&$2905$&$11182$&$253892$&$3616144$&$56742775$ \\
$8$& $64$&$-230$&$2438$&$-5135$&$9917$&$-124467$&$-1306559$&$-25952916$&$-459328520$ \\
$9$& $128$&$-734$&$8721$&$-25446$&$112518$&$47564$&$9957185$&$178180002$&$3726950202$ \\
$10$& $256$&$-2289$&$31602$&$-134879$&$552339$&$-4683027$&$-42459898$&$-1308692710$&$-29852809180$ \\
$11$& $512$&$-7039$&$116821$&$-670008$&$4149475$&$-9296152$&$399601667$&$8669878028$&$242171554559$ \\
$12$& $1024$&$-21460$&$437758$&$-3414254$&$22844193$&$-188955568$&$-1233422911$&$-65810827609$&$-1924630979085$ \\
$13$& $2048$&$-65064$&$1663481$&$-17022549$&$149941792$&$-725537667$&$16798999974$&$416483532392$&$15652884150733$ \\
$14$& $4096$&$-196559$&$6388202$&$-85672220$&$864112247$&$-8133442381$&$-25768675818$&$-3327304711052$&$-123449090799389$ \\
$15$& $8192$&$-592409$&$24759741$&$-427885725$&$5376583485$&$-41758325066$&$753773677302$&$19660027898985$&$1009591434254489$ \\
$16$& $16384$&$-1782690$&$96647478$&$-2144390153$&$31618003029$&$-367416474589$&$247507159657$&$-170405333570573$&$-7887831186821342$ \\
\end{tabular}
\caption{Euler characteristic $e(\MG_{g,n})$}
\label{tab:eeven}
\end{table}

\begin{table}
\begin{tabular}{c|r r r r r r r r r}
\backslashbox{$n$}{$g$}& 1&2&3&4&5&6&7&8&9 \\
\hline
$0$& &$0$&$0$&$-1$&$0$&$-1$&$-2$&$-8$&$-38$ \\
$1$& $0$&$0$&$-1$&$-1$&$-1$&$1$&$3$&$34$&$278$ \\
$2$& $-1$&$0$&$-3$&$-3$&$-5$&$-7$&$-45$&$-273$&$-2143$ \\
$3$& $-2$&$0$&$-8$&$-7$&$-9$&$20$&$172$&$1688$&$16279$ \\
$4$& $-4$&$-1$&$-24$&$-27$&$-66$&$-184$&$-1365$&$-12037$&$-127665$ \\
$5$& $-8$&$-5$&$-71$&$-81$&$-137$&$340$&$6153$&$80002$&$995425$ \\
$6$& $-16$&$-20$&$-228$&$-364$&$-1185$&$-5356$&$-47862$&$-568014$&$-7861828$ \\
$7$& $-32$&$-70$&$-743$&$-1394$&$-3536$&$3081$&$210751$&$3814347$&$61973273$ \\
$8$& $-64$&$-231$&$-2544$&$-6716$&$-28509$&$-166300$&$-1760727$&$-27489461$&$-492911760$ \\
$9$& $-128$&$-735$&$-8891$&$-29974$&$-118323$&$-144896$&$6985727$&$182953976$&$3900012883$ \\
$10$& $-256$&$-2290$&$-32028$&$-148035$&$-837525$&$-5638835$&$-67361954$&$-1348453174$&$-31209047062$ \\
$11$& $-512$&$-7040$&$-117503$&$-708621$&$-4206038$&$-13846788$&$214832720$&$8779760089$&$246950368646$ \\
$12$& $-1024$&$-21461$&$-439464$&$-3528385$&$-27300597$&$-211328036$&$-2704794269$&$-66861595436$&$-1987593802313$ \\
$13$& $-2048$&$-65065$&$-1666211$&$-17361527$&$-150529289$&$-834038152$&$5420792406$&$418887979427$&$15685889711601$ \\
$14$& $-4096$&$-196560$&$-6395028$&$-86682326$&$-934471725$&$-8667253394$&$-115615170716$&$-3355504362873$&$-127079507192857$ \\
$15$& $-8192$&$-592410$&$-24770663$&$-430902388$&$-5383163340$&$-44375995137$&$48464678629$&$19709309010324$&$997347475403633$ \\
$16$& $-16384$&$-1782691$&$-96674784$&$-2153412834$&$-32735332605$&$-380341928592$&$-5327400148291$&$-171167953029982$&$-8150861890475894$ \\
\end{tabular}
\caption{Euler characteristics $e^\mathrm{odd}(\MG_{g,n})$}
\label{tab:eodd}
\end{table}

\begin{table}
\begin{tabular}{c|r r r r r r r r r r r r r r r r}
\backslashbox{$n$}{$g$}& 1&2&3&4&5&6&7&8&9&10&11&12&13&14&15&16 \\
\hline
$0$& &$1$&$1$&$2$&$1$&$2$&$1$&$1$&$-21$&$-124$&$-1202$&$-10738$&$-112901$&$-1271148$&$-15668391$&$-208214777$ \\
$1$& $1$&$1$&$1$&$2$&$0$&$3$&$3$&$31$&$154$&$1405$&$12409$&$128198$&$1428208$&$17431842$&$229796854$&$3260731764$ \\
$2$& $1$&$1$&$2$&$2$&$0$&$1$&$-8$&$-71$&$-645$&$-5916$&$-60661$&$-680524$&$-8319674$&$-110000218$&$-1564363190$&$-23810497027$ \\
$3$& $1$&$1$&$2$&$2$&$1$&$4$&$22$&$148$&$1432$&$14933$&$173615$&$2170285$&$29302324$&$423870178$&$6547649971$&$107566687178$ \\
$4$& $1$&$1$&$2$&$1$&$2$&$2$&$-6$&$-158$&$-1911$&$-24310$&$-324161$&$-4610023$&$-69547908$&$-1112576568$&$-18825687993$&$-336214369334$ \\
$5$& $1$&$1$&$2$&$1$&$2$&$4$&$14$&$114$&$1677$&$26129$&$415109$&$6837661$&$117254253$&$2100407688$&$39335548893$&$770173833141$ \\
$6$& $1$&$1$&$2$&$1$&$2$&$3$&$9$&$-34$&$-900$&$-18622$&$-368380$&$-7231581$&$-143971474$&$-2934457032$&$-61596805980$&$-1335762223221$ \\
$7$& $1$&$1$&$2$&$1$&$2$&$3$&$11$&$8$&$296$&$8548$&$223615$&$5456280$&$129717602$&$3072060291$&$73437715819$&$1786143030025$ \\
$8$& $1$&$1$&$2$&$1$&$2$&$3$&$11$&$0$&$-20$&$-2292$&$-88814$&$-2878337$&$-85129652$&$-2410044793$&$-66988401529$&$-1856600124025$ \\
$9$& $1$&$1$&$2$&$1$&$2$&$3$&$11$&$0$&$18$&$268$&$20992$&$1010696$&$39693125$&$1399384996$&$46526543233$&$1500185772169$ \\
$10$& $1$&$1$&$2$&$1$&$2$&$3$&$11$&$0$&$18$&$-7$&$-2154$&$-212879$&$-12481703$&$-584609581$&$-24221345542$&$-934645878596$ \\
$11$& $1$&$1$&$2$&$1$&$2$&$3$&$11$&$0$&$18$&$-7$&$71$&$20256$&$2377198$&$166410659$&$9165254242$&$440926365563$ \\
$12$& $1$&$1$&$2$&$1$&$2$&$3$&$11$&$0$&$18$&$-7$&$71$&$-102$&$-207026$&$-28931603$&$-2383055405$&$-152471494842$ \\
$13$& $1$&$1$&$2$&$1$&$2$&$3$&$11$&$0$&$18$&$-7$&$71$&$-102$&$295$&$2319534$&$381124388$&$36485839046$ \\
$14$& $1$&$1$&$2$&$1$&$2$&$3$&$11$&$0$&$18$&$-7$&$71$&$-102$&$295$&$-602$&$-28285686$&$-5402710802$ \\
$15$& $1$&$1$&$2$&$1$&$2$&$3$&$11$&$0$&$18$&$-7$&$71$&$-102$&$295$&$-602$&$1730$&$373201519$ \\
$16$& $1$&$1$&$2$&$1$&$2$&$3$&$11$&$0$&$18$&$-7$&$71$&$-102$&$295$&$-602$&$1730$&$-3616$ \\
\end{tabular}
\caption{Euler characteristics $e(\MG_{g,n}^{\SG_n})$}
\label{tab:eEeven}
\end{table}

\begin{table}
\begin{tabular}{c|r r r r r r r r r r r r r r r r}
\backslashbox{$n$}{$g$}& 1&2&3&4&5&6&7&8&9&10&11&12&13&14&15&16 \\
\hline
$0$& &$0$&$0$&$-1$&$0$&$-1$&$-2$&$-8$&$-38$&$-275$&$-2225$&$-20358$&$-207321$&$-2320136$&$-28287416$&$-373205135$ \\
$1$& $0$&$0$&$-1$&$-1$&$-1$&$1$&$3$&$34$&$278$&$2285$&$20921$&$212777$&$2376903$&$28931001$&$381122658$&$5402707186$ \\
$2$& $0$&$1$&$-1$&$-1$&$-2$&$-1$&$-17$&$-114$&$-918$&$-8555$&$-88885$&$-1010798$&$-12481998$&$-166411261$&$-2383057135$&$-36485842662$ \\
$3$& $0$&$1$&$-2$&$-1$&$-2$&$2$&$11$&$158$&$1659$&$18615$&$223544$&$2878235$&$39692830$&$584608979$&$9165252512$&$152471491226$ \\
$4$& $0$&$1$&$-2$&$1$&$-1$&$0$&$-19$&$-148$&$-1929$&$-26136$&$-368451$&$-5456382$&$-85129947$&$-1399385598$&$-24221347272$&$-440926369179$ \\
$5$& $0$&$1$&$-2$&$1$&$-2$&$1$&$-8$&$71$&$1414$&$24303$&$415038$&$7231479$&$129717307$&$2410044191$&$46526541503$&$934645874980$ \\
$6$& $0$&$1$&$-2$&$1$&$-2$&$3$&$-10$&$-31$&$-663$&$-14940$&$-324232$&$-6837763$&$-143971769$&$-3072060893$&$-66988403259$&$-1500185775785$ \\
$7$& $0$&$1$&$-2$&$1$&$-2$&$3$&$-11$&$-1$&$136$&$5909$&$173544$&$4609921$&$117253958$&$2934456430$&$73437714089$&$1856600120409$ \\
$8$& $0$&$1$&$-2$&$1$&$-2$&$3$&$-11$&$0$&$-39$&$-1412$&$-60732$&$-2170387$&$-69548203$&$-2100408290$&$-61596807710$&$-1786143033641$ \\
$9$& $0$&$1$&$-2$&$1$&$-2$&$3$&$-11$&$0$&$-18$&$117$&$12338$&$680422$&$29302029$&$1112575966$&$39335547163$&$1335762219605$ \\
$10$& $0$&$1$&$-2$&$1$&$-2$&$3$&$-11$&$0$&$-18$&$-7$&$-1273$&$-128300$&$-8319969$&$-423870780$&$-18825689723$&$-770173836757$ \\
$11$& $0$&$1$&$-2$&$1$&$-2$&$3$&$-11$&$0$&$-18$&$-7$&$-71$&$10636$&$1427913$&$109999616$&$6547648241$&$336214365718$ \\
$12$& $0$&$1$&$-2$&$1$&$-2$&$3$&$-11$&$0$&$-18$&$-7$&$-71$&$-102$&$-113196$&$-17432444$&$-1564364920$&$-107566690794$ \\
$13$& $0$&$1$&$-2$&$1$&$-2$&$3$&$-11$&$0$&$-18$&$-7$&$-71$&$-102$&$-295$&$1270546$&$229795124$&$23810493411$ \\
$14$& $0$&$1$&$-2$&$1$&$-2$&$3$&$-11$&$0$&$-18$&$-7$&$-71$&$-102$&$-295$&$-602$&$-15670121$&$-3260735380$ \\
$15$& $0$&$1$&$-2$&$1$&$-2$&$3$&$-11$&$0$&$-18$&$-7$&$-71$&$-102$&$-295$&$-602$&$-1730$&$208211161$ \\
$16$& $0$&$1$&$-2$&$1$&$-2$&$3$&$-11$&$0$&$-18$&$-7$&$-71$&$-102$&$-295$&$-602$&$-1730$&$-3616$ \\
\end{tabular}
\caption{Euler characteristics $e^\mathrm{odd}(\MG_{g,n}^{\SG_n})$}
\label{tab:eEodd}
\end{table}

\begin{table}
\begin{tabular}{c|r r r r r r}
\backslashbox{$g$}{$n$}& 0&1&2&3&4&5 \\
\hline
$0$& &&&$s_{3}$&$s_{4}$&$s_{5}$ \\
$1$& &$s_{1}$&$s_{2}$&$s_{1,1,1} + s_{3}$&$s_{2,1,1} + s_{4}$&$s_{1,1,1,1,1} + s_{3,1,1} + s_{5}$ \\
$2$& $1$&$s_{1}$&$s_{2}$&$s_{3}$&$-s_{2,1,1} + s_{2,2} + s_{4}$&$-s_{2,1,1,1} - s_{3,1,1} + s_{3,2} + s_{5}$ \\
$3$& $1$&$s_{1}$&$2 s_{2}$&$2 s_{2,1} + 2 s_{3}$&$s_{2,1,1} + 2 s_{2,2} + 3 s_{3,1} + 2 s_{4}$&$s_{2,1,1,1} + 2 s_{2,2,1} + 3 s_{3,1,1} + 3 s_{3,2} + 3 s_{4,1} + 2 s_{5}$ \\
$4$& $2 $&$2 s_{1}$&$s_{1,1} + 2 s_{2}$&$2 s_{1,1,1} + s_{2,1} + 2 s_{3}$&$s_{1,1,1,1} + s_{2,2} + s_{4}$&$2 s_{1,1,1,1,1} + s_{2,1,1,1} - s_{3,1,1} - s_{4,1} + s_{5}$ \\
$5$& $1$&$0$&$-2 s_{1,1}$&$-s_{1,1,1} + s_{3}$&$-5 s_{1,1,1,1} - 4 s_{2,1,1} - 3 s_{2,2} - s_{3,1} + 2 s_{4}$&$4 s_{2,1,1,1} + s_{2,2,1} + 4 s_{3,1,1} + s_{3,2} + 3 s_{4,1} + 2 s_{5}$ \\
$6$& $2 $&$3 s_{1}$&$s_{2}$&$10 s_{1,1,1} + 11 s_{2,1} + 4 s_{3}$&$-5 s_{1,1,1,1} - 15 s_{2,1,1} - 4 s_{2,2} - 10 s_{3,1} + 2 s_{4}$&$34 s_{1,1,1,1,1} + 51 s_{2,1,1,1} + 41 s_{2,2,1} + 22 s_{3,1,1} + 16 s_{3,2} + 4 s_{4,1} + 4 s_{5}$ \\
\end{tabular}
\caption{$\SG_n$-equivariant Euler characteristic $e_{\SG_n}(\MG_{g,n})$ in the Schur basis of $\Lambda_n$}
\label{tab:eEschureven}
\end{table}

\begin{table}
\begin{tabular}{c|r r r r r r}
\backslashbox{$g$}{$n$}& 0&1&2&3&4&5 \\
\hline
$0$& &&&$s_{3}$&$s_{4}$&$s_{5}$ \\
$1$& &$0$&$-s_{1,1}$&$-s_{2,1}$&$-s_{1,1,1,1} - s_{3,1}$&$-s_{2,1,1,1} - s_{4,1}$ \\
$2$& $0$&$0$&$-s_{1,1} + s_{2}$&$-s_{1,1,1} + s_{3}$&$-s_{1,1,1,1} - s_{2,1,1} + s_{2,2} + s_{4}$&$-s_{1,1,1,1,1} - s_{2,1,1,1} - s_{3,1,1} + s_{3,2} + s_{5}$ \\
$3$& $0$&$-s_{1}$&$-2 s_{1,1} - s_{2}$&$-2 s_{1,1,1} - 2 s_{2,1} - 2 s_{3}$&$-2 s_{1,1,1,1} - 3 s_{2,1,1} - s_{2,2} - 3 s_{3,1} - 2 s_{4}$&$-2 s_{1,1,1,1,1} - 3 s_{2,1,1,1} - 2 s_{2,2,1} - 3 s_{3,1,1} - 3 s_{3,2} - 3 s_{4,1} - 2 s_{5}$ \\
$4$& $-1$&$-s_{1}$&$-2 s_{1,1} - s_{2}$&$-2 s_{1,1,1} - 2 s_{2,1} - s_{3}$&$-3 s_{1,1,1,1} - 4 s_{2,1,1} - 2 s_{2,2} - 3 s_{3,1} + s_{4}$&$-3 s_{1,1,1,1,1} - 4 s_{2,1,1,1} - 4 s_{2,2,1} - 4 s_{3,1,1} - 3 s_{3,2} - s_{4,1} + s_{5}$ \\
$5$& $0$&$-s_{1}$&$-3 s_{1,1} - 2 s_{2}$&$-s_{1,1,1} - 3 s_{2,1} - 2 s_{3}$&$-7 s_{1,1,1,1} - 9 s_{2,1,1} - 5 s_{2,2} - 7 s_{3,1} - s_{4}$&$-2 s_{1,1,1,1,1} - 6 s_{2,1,1,1} - 5 s_{2,2,1} - 8 s_{3,1,1} - 4 s_{3,2} - 4 s_{4,1} - 2 s_{5}$ \\
$6$& $-1$&$s_{1}$&$-6 s_{1,1} - s_{2}$&$4 s_{1,1,1} + 7 s_{2,1} + 2 s_{3}$&$-22 s_{1,1,1,1} - 32 s_{2,1,1} - 12 s_{2,2} - 14 s_{3,1}$&$19 s_{1,1,1,1,1} + 27 s_{2,1,1,1} + 26 s_{2,2,1} + 6 s_{3,1,1} + 10 s_{3,2} - s_{4,1} + s_{5}$ \\
\end{tabular}
\caption{$\SG_n$-equivariant Euler characteristic $e_{\SG_n}^\mathrm{odd}(\MG_{g,n})$ in the Schur basis of $\Lambda_n$}
\label{tab:eEschurodd}
\end{table}

\end{landscape}
\restoregeometry

\newcommand{\etalchar}[1]{$^{#1}$}
\providecommand{\bysame}{\leavevmode\hbox to3em{\hrulefill}\thinspace}
\providecommand{\MR}{\relax\ifhmode\unskip\space\fi MR }
\providecommand{\MRhref}[2]{%
  \href{http://www.ams.org/mathscinet-getitem?mr=#1}{#2}
}
\providecommand{\href}[2]{#2}

\end{document}